\documentclass[12pt, oneside, reqno]{amsart}

\usepackage[T1]{fontenc}
\usepackage[margin=2.5cm]{geometry}
\usepackage{todonotes}
\usepackage{amsmath,amssymb,amsthm,amsfonts,amsaddr}
\usepackage{mathtools,mathrsfs}
\usepackage{graphicx}
\usepackage{hyperref}
\usepackage{xcolor}
\usepackage{tikz, tikz-cd}
\usepackage{pgfplots}
\usepackage{pgfplotstable}
\usepackage{subcaption}
\usepackage{booktabs}
\newtheorem{theorem}{Theorem}[section]
\newtheorem{lemma}[theorem]{Lemma}

\newtheorem{proposition}[theorem]{Proposition}

\theoremstyle{remark}
\newtheorem{remark}[theorem]{Remark}
\newtheorem{example}[theorem]{Example}
\newtheorem{assumption}[theorem]{Assumption}

\title[Estimators for eigenspace error]{
  Reliable eigenspace error estimation using source error estimators}
\author{Jay Gopalakrishnan \and Gabriel Pinochet-Soto}

\dedicatory{Dedicated to the memory of Professor Raytcho Lazarov}

\pgfplotsset{compat=1.18}

\pgfplotsset{
    compat=newest,
    table/col sep=comma,
    width=0.75\textwidth,
    height=0.75\textwidth,
    grid=both,
    legend pos=north east,
    minor tick num=1,
    major grid style={line width=.2pt,draw=gray!30},
    minor grid style={line width=.2pt,draw=gray!15},
}

\pgfplotscreateplotcyclelist{three-by-three}{%
    {orange!90!black, mark=diamond, solid, thick},
    {orange!90!black, mark=triangle, solid, thick},
    {orange!90!black, mark=square, solid, thick},
    {blue!90!black, mark=diamond, dashed, thick},
    {blue!90!black, mark=triangle, dashed, thick},
    {blue!90!black, mark=square, dashed, thick},
    {purple!90!black, mark=diamond, dotted, thick},
    {purple!90!black, mark=triangle, dotted, thick},
    {purple!90!black, mark=square, dotted, thick},
}

\pgfplotscreateplotcyclelist{three-by-two}{%
    {red, mark=diamond, solid, thick},
    {blue, mark=triangle, solid, thick},
    {orange, mark=square, solid, thick},
    {red!80!black, mark=diamond*, thick},
    {blue!80!black, mark=triangle*, thick},
    {orange!80!black, mark=square*, thick},
}

\pgfplotscreateplotcyclelist{two-by-two}{%
    {orange!90!black, mark=diamond, solid, thick},
    {orange!90!black, mark=triangle, solid, thick},
    {blue!90!black, mark=diamond, dashed, thick},
    {blue!90!black, mark=triangle, dashed, thick},
}

\pgfplotscreateplotcyclelist{one-by-two}{%
    {orange, mark=square, solid, thick},
    {orange!80!black, mark=square*, thick},
}

\pgfplotscreateplotcyclelist{two-in-one}{%
    {orange!90!black, mark=diamond, solid, ultra thick},
    {blue!90!black, mark=square, densely dotted, thick},
}

\newcommand{\comment}[1]{}

\newcommand{\ii}{\mathfrak{i}}
\newcommand{\diff}{\mathop{}\!\mathrm{d}}

\newcommand{\tForAll}{\text{ for all }}
\newcommand{\tAnd}{\text{ and }}

\renewcommand{\hat}{\widehat}

\DeclareMathOperator{\diver}{div}

\DeclareMathOperator{\ran}{ran}
\DeclareMathOperator{\dom}{dom}

\DeclareMathOperator{\spn}{span}

\DeclareMathOperator{\dist}{dist}

\DeclareMathOperator{\gap}{gap}
\DeclareMathOperator{\osc}{osc}

\newcommand{\RR}{\mathbb{R}}
\newcommand{\CC}{\mathbb{C}}

\DeclarePairedDelimiter{\abs}{\lvert}{\rvert}
\DeclarePairedDelimiter{\norm}{\lVert}{\rVert}

\DeclarePairedDelimiter{\set}{\{}{\}}
\DeclarePairedDelimiter{\parens}{(}{)}

\newcommand{\nquad}{N}

\newcommand{\vL}{\varLambda}
\newcommand{\vT}{\varTheta}
\newcommand{\vU}{\varUpsilon}
\newcommand{\vG}{\varGamma}

\newcommand{\dz}{\diff{z}}
\newcommand{\Srh}{S_{r, h}}
\newcommand{\Prh}{P_{r, h}}
\newcommand{\EE}{\mathcal{E}}
\newcommand{\Ac}{\mathcal{A}}
\newcommand{\Af}{{A}}
\newcommand{\Hc}{\mathcal{H}}
\newcommand{\og}{\omega}
\newcommand{\veps}{\varepsilon}
\newcommand{\vpi}{\varPi}
\newcommand{\om}{\Omega}
\newcommand{\oh}{\Omega_h}
\newcommand{\Ho}{\mathring{H}}
\newcommand{\ip}[1]{\langle{#1}\rangle}
\newcommand{\Cregz}{C_{z,\text{reg}}}
\DeclareFontFamily{U}{boondoxuprscr}{\skewchar \font =45}
\DeclareFontShape{U}{boondoxuprscr}{m}{n}{
    <-> BOONDOXUprScr-Regular}{}
\DeclareFontShape{U}{boondoxuprscr}{b}{n}{
    <-> BOONDOXUprScr-Bold}{}
\newcommand\hcal{\text{\fontencoding{U}\fontfamily{boondoxuprscr}\fontshape{n}\selectfont h}}
\newcommand\scal{\text{\fontencoding{U}\fontfamily{boondoxuprscr}\fontshape{n}\selectfont s}}

\begin{document}

\begin{abstract}
    We introduce a framework for repurposing error estimators for source
  problems to compute an estimator for the gap between eigenspaces and their
  discretizations. Of interest are eigenspaces of finite
  clusters of eigenvalues of unbounded nonselfadjoint linear operators
  with compact resolvent.  Eigenspaces and eigenvalues of rational
  functions of such operators are studied as a  first step.
  Under an assumption of convergence of resolvent approximations in
  the operator norm and an assumption on global reliability of source
  problem error estimators, we show that the gap in eigenspace
  approximations can be bounded by a globally reliable and computable
  error estimator. Also included are applications of the theoretical
  framework to first-order system least squares
  (FOSLS) discretizations and discontinuous Petrov-Galerkin (DPG)
  discretizations, both yielding new estimators for the error gap.
  Numerical experiments with a selfadjoint model problem and with a
  leaky nonselfadjoint  waveguide eigenproblem show that adaptive algorithms
  using the new estimators give refinement patterns
  that target the cluster as a whole instead of  individual eigenfunctions.


\end{abstract}

\maketitle

\section{Introduction} 

When clusters of some physically relevant eigenvalues are to be approximated,
eigensolvers  often use
functions on  the complex plane to transform such eigenvalues to a dominant set of
eigenvalues of the rational function of the operator. Shifted inverse
iteration is perhaps the simplest example. Contour
integral iterative solvers provide more involved examples. Such eigensolvers
typically require many applications of the resolvent operator, which amounts
to solving ``source problems.''  When using
finite element methods for such source problems, one often has ready
access to {\em a posteriori} error estimators for source problems.
The goal of this work
is to repurpose such source problem error estimators to give error
indicators for the gap between discrete and exact eigenspaces of the
targeted eigenvalue cluster.

Error estimation for eigenvalue problems is well-studied.  In the
extensive literature on this topic, one finds estimators for
least-squares discretizations~\cite{BerBof2022-1, BerBof2020-1},
guaranteed bounds for eigenvalue clusters of selfadjoint
problems~\cite{CanEtAl2017-1, CanEtAl2020-1}, homotopy methods for
nonselfadjoint problems~\cite{CarstensenEtAl2011-1}, Crouzeix-Raviart
discretizations for the Laplace operator~\cite{CartensenEtAl2015-1,%
CarGed2014-1}, mixed methods~\cite{Boffi2010-1}, DPG
method~\cite{BerBofSch2023-1}, skeletal finite
element methods~\cite{CarZhaZha2020-1}, oscillation-free
discretizations for symmetric operators~\cite{CarGed2011-1}, dual
weighted residual methods~\cite{HeuRan2001-1, RanWesWol2010-1},
explicit error estimators~\cite{Lar2000-1}, and works targeting
cluster error estimation~\cite{CanEtAl2020-1, 
  GianiEtAl2012-1, GianiEtAl2016-1}. Source problem error estimators
have of course been studied even longer, see e.g., \cite{AinOde1997-1, BabRhe1979-1,
  BabRhe1978-1, DieJosHue1998-1, HakNeuOva2017-1, NagZhaZho2006-1, ZieZhu1992-1,
  ZieZhu1992-2}.  The goal of this
paper is to make it easy to reutilize source problem error estimators
within an eigensolver.  While our basic idea for eigenspace error
estimation is not radically different from the cited previous
literature, our focus on unbounded nonselfadjoint operators within the
theory presented here is new, and so is the practical perspective of repurposing
source problem error estimators within often-used spectral mappings by
rational functions.

In various practical scenarios, such as in our prior work on eigenmodes of
microstructured optical fibers~\cite{GopGroPinVan2025-1}, the need
arises to numerically approximate the eigenspace corresponding to a
finite cluster $\vL$ of physically interesting eigenvalues in the
spectrum of an unbounded differential operator $\Ac$.  These
eigenvalues need neither be dominant nor lie near any extremity of
the spectrum.  Consider eigensolvers that target  such clusters
using a rational function $r(z)$ of a complex variable~$z$, applied to
the operator $\Ac$. The $r(\cdot)$ is usually designed to make
$r(\vL)$ into a dominant set of eigenvalues of $r(\Ac)$ so that
standard iterative techniques can capture them easily. For example,
the subspace iteration $E_{(\ell)} = r(\Ac) E_{(\ell-1)}$ for
$\ell=1,2,\dots,$ starting with some initial space $E_{(0)}$, can be
proven to converge to the wanted eigenspace under suitable
conditions~\cite{GopGruOva2020-1}.
Of course, in practice, we must apply such eigensolvers to
finite-dimensional approximations of $r(\Ac)$, such as finite element
approximations, which then capture a discretized eigenspace. We are
interested in computable error estimators for such discrete
eigenspaces that can be used as a basis for automatic adaptive
algorithms. A key difference when compared to the standard adaptivity
for source problems is that here the desired output is a {\em space},
a multidimensional eigenspace, rather than a single solution function.
Accordingly, the reliability of our error estimator is established using the gap metric which gives a notion of distance between spaces. The main result,
Theorem~\ref{thm:espace-gap-estimator}, bounds the gap between exact and approximate eigenspace by a computable error estimator.

A prominent example of the above-mentioned type of rational functions
(that make the targeted portion of the spectrum into a dominant one)
is the so-called ``FEAST method''~\cite{GopGruOva2020-1, KesPolPet2016-1, PetPol2014-1, Pol2009-1}.  It is a filtered subspace
iteration, i.e., a subspace iteration using $r(\Ac)$ where $r(\cdot)$
is viewed as a filter that isolates wanted portions of the spectrum.
In such examples, it is essential to understand how the eigenspaces of
$r(\Ac)$ relate to those of $\Ac$ since the subspace iteration
captures the former.  This requires us to tailor the spectral mapping
theorem~\cite{HilPhi1996-1} to the case of a rational function of a
linear operator.  We provide elementary arguments clarifying how
eigenspaces are transformed under $r(\cdot)$ by building on a result of
Yamamoto~\cite{Yam1971-1} (which he proves very simply using
B\'ezout's identity). This first step (completed in the next section)
then clarifies what assumptions we must place on the rational filters
in the remainder of the analysis for nonselfadjoint~$\Ac$.  We then
complete a general
framework for the design of source-problem error estimators that can
be used for eigenspace error estimation. The assumptions in the
framework are verified for two finite element discretizations of a
model problem.

This work is organized as follows.
In Section~\ref{sec:rational-functions}, we concentrate our attention on the
spectral mapping theorem for rational functions of a linear operator.
Section~\ref{sec:eigenspace-error-estimate} is devoted to the eigenspace
error estimation problem from a filtered subspace iteration.
In this section, we introduce the error representation operator, which is the
key to the integration of source-problem error estimators with eigenspace
error estimation.
In Subsections~\ref{ssec:FOSLS} and~\ref{ssec:DPG}, we present two concrete
applications of the error representation operator to the first-order system
least squares (FOSLS)~\cite{CaiEtAl1994-1, CaiEtAl1997-1}
and the discontinuous Petrov-Galerkin (DPG)~\cite{DemGop2013-1, DemGop2017-1, DemGop2025-1}
discretizations of the finite element method, respectively.
Both these methods, due to their nature as a residual
minimization method, are able to provide intrinsic residual-based error estimators.
In Section~\ref{sec:numerical-results}, we present numerical results
that illustrate the performance of the proposed error estimators for a
model problem.

\section{Rational functions of an operator} \label{sec:rational-functions}

We study how the generalized eigenspaces of a rational function of a
linear operator are related to those  of the operator.
The main result of this section is Lemma~\ref{lem:rational-gen-eigenspaces} which shows how algebraic eigenspaces are transformed by rational functions of the operator.

Let \(\Hc\) be a complex ($\CC$) Hilbert space, endowed with norm
\(\norm{\cdot}_\Hc\). Let $\Ac : \dom (\Ac) \subseteq \Hc \to \Hc$ be a
closed linear operator (that is not assumed to be bounded or
selfadjoint).  We use $\sigma(\cdot)$ and $\rho(\cdot)$ to denote the
spectrum and resolvent set of an operator, respectively.  Recall that
a complex number $\lambda$ in the spectrum $\sigma(\Ac)$ is called an
eigenvalue of $\Ac$ if $\ker( \lambda - \Ac)$ is
nontrivial. Throughout, in expressions like $\lambda - \Ac$, the
complex number $\lambda$ also serves to denote the linear operator
that scales vectors by it, and
$\ker(\cdot)$ denotes the null space.  

Consider
complex-valued rational functions $r$ of the form
\begin{equation}
    r(z) =  \og_0 +  \sum_{j=1}^\nquad 
    \sum_{i=1}^{\nu_j}\frac{\omega_{j, i}}{(z_j - z)^i}
    \label{eq:feast-filter}
\end{equation}
for some $\og_0, \omega_{j,i} \in \CC$, some distinct $z_j \in
\rho(\Ac)$, and some positive integers $\nu_j$, $j=1,\dots,N$.
Substituting $z$ by $\Ac$, we define the bounded
linear operator $r(\Ac)$ on $\Hc$, namely
\begin{gather*}
  r(\Ac) =
  \og_0 +
  \sum_{j=1}^\nquad \sum_{i=1}^{\nu_j}\omega_{j, i} R(z_j)^i.
\end{gather*}
Here and throughout, for each complex number $z$ in the resolvent set, we denote
by $R(z)$ the resolvent operator of $\Ac$ that maps from $\Hc$ into $\dom(\Ac)$,
given by $R(z) f= (z - \Ac)^{-1}f$ for any $f \in \Hc$.

A result of~\cite{Yam1971-1} shows that for any set of distinct
complex numbers $\alpha_i$ and positive integers $n_i$, $i=1,\dots,s$,
the equality
\begin{equation}
  \label{eq:Yamamoto-decomposition}
    \ker \left( \prod_{i=1}^s (\alpha_i - \Ac)^{n_i} \right)
  = \bigoplus_{i=1}^s \ker ( \alpha_i - \Ac)^{n_i},
\end{equation}
holds, where the right hand side is a direct sum, when
$\dom (\Ac) = \Hc$. It was noted in \cite{GonzaOniev86}
that~\eqref{eq:Yamamoto-decomposition} also holds when $\dom (\Ac)$ is
smaller than $\Hc$ and equals the domain of polynomials in $\Ac$ of
degree $d= \sum_{i=1}^s n_i$. We use this result to study how
generalized eigenvectors of $r(\Ac)$ are related to those of $\Ac$
(for unbounded $\Ac$ with general domains).

Let $\dom(\Ac^1) = \dom(\Ac)$, and recursively define 
$\dom(\Ac^n ) = \{ x \in \dom (\Ac) : \Ac x \in \dom
(\Ac^{n-1})\}$ for any $n > 1$.
Then set
\begin{gather*}
  \dom (\Ac^\infty)= \bigcap_{n=1}^\infty \dom (\Ac^n).
\end{gather*}
The generalized eigenspace of the unbounded $\Ac$ corresponding to an 
eigenvalue $\lambda$ is
\[
  E_\lambda^\infty(\Ac) = \bigcup_{n=1}^\infty \{ u \in \dom (\Ac^\infty) :
  (\Ac-\lambda)^n u = 0   \}. 
\]
We now provide an elementary argument to  relate this space  to the generalized
eigenspace of the bounded
linear operator $r(\Ac): \Hc \to \Hc$ corresponding to an eigenvalue
$\mu \in \CC$, namely
\[
  E_\mu(r(\Ac)) =
  \bigcup_{n=1}^\infty \{ u \in \Hc: ( r(\Ac) -\mu )^n u = 0\}.
\]
Let
\begin{equation}
  \label{eq:r-mu-inv}
  r^{-1}_\mu = \{ \lambda \in \CC \setminus \{ z_1, \dots, z_N\}:  \,
  \mu = r(\lambda)\}  
\end{equation}
for any $0 \ne \mu \in \CC$.  This set of inverse images of $\mu$ under
$r$ is used in the next lemma.

\begin{lemma}
  \label{lem:rational-gen-eigenspaces}
  For any $\mu\in \CC$ that does not coincide with  $\og_0$,
  \begin{equation}
    \label{eq:13}    
    E_\mu(r(\Ac)) = \bigoplus_{\lambda \,\in\, r^{-1}_\mu } E_\lambda^\infty (\Ac).
  \end{equation}
\end{lemma}
  
\begin{proof}
  Letting 
  \[
    E_\mu^\infty(r(\Ac)) =
    \bigcup_{n=1}^\infty \{ u \in \dom (\Ac^\infty) : (r(\Ac) - \mu )^n u = 0\},
  \]
  we claim that
  \begin{equation}
    \label{eq:9}
    E_\mu^\infty(r(\Ac)) = E_\mu(r(\Ac)). 
  \end{equation}
  Obviously $E_\mu^\infty(r(\Ac)) \subseteq E_\mu(r(\Ac)).$ To prove
  the reverse inclusion, we begin by proving that 
  \begin{equation}
    \label{eq:10}
    \{ u \in \Hc: \; (r(\Ac) - \mu) u \in \dom (\Ac^\infty) \}
    \subseteq  \dom (\Ac^\infty).
  \end{equation}
  Indeed, given any $u \in \Hc$, if $v = (r(\Ac) - \mu ) u$ is in
  $\dom(\Ac^\infty)$, then
  \begin{equation}
    \label{eq:11}
    u = \frac{1}{\og_0-\mu}\left( v -
      \sum_{j=1}^\nquad \sum_{i=1}^{\nu_j}\omega_{j, i} R(z_j)^i
      u\right)
  \end{equation}
  must be in $\dom (\Ac)$ since the $R(z_j)$ maps into
  $\dom (\Ac)$ and $v$ is in $\dom (\Ac)$. Since $z_j \in \rho(\Ac)$ and $u \in\dom(\Ac)$, the resolvent commutes with the operator when applied to $u$, i.e.,
  $R(z_j) \Ac u = \Ac  R(z_j) u$.
  Both sides of this equality are in $\dom (\Ac)$, so repeatedly commuting,
  we find that
  $R(z_j)^i \Ac u = \Ac  R(z_j)^i u$ for higher powers $i$ as well.  
  Hence, applying $\Ac$ to both sides of~\eqref{eq:11},
  \begin{equation}
    \label{eq:12}
    \Ac u = \frac{1}{\og_0 - \mu}\left( \Ac v -
      \sum_{j=1}^\nquad \sum_{i=1}^{\nu_j}\omega_{j, i} R(z_j)^i
      \Ac u\right).
  \end{equation}
  This shows  that $\Ac u \in \dom (\Ac)$, i.e., $u \in \dom
  (\Ac^2)$. Continuing this argument, we find that
  $u \in \dom (\Ac^\infty)$, thus proving~\eqref{eq:10}.

  Next, observe that for any integer $k \ge 1$ and any $u \in \Hc$, if
  $(r(\Ac) - \mu)^ku$ is in $\dom (\Ac^\infty)$, then $u$ must also be
  in $ \dom(\Ac^\infty)$. When $k=1$, we just proved this
  in~\eqref{eq:10}. When $k>1$, we use induction: since
  $(r(\Ac) - \mu) (r(\Ac) - \mu)^{k-1}u$ is in $\dom (\Ac^\infty)$,
  by~\eqref{eq:10} we conclude that $(r(\Ac) - \mu)^{k-1}u$ is in
  $\dom (\Ac^\infty)$, so by induction hypothesis, $u$ is in
  $\dom (\Ac^\infty)$.  In particular, if $u \in \Hc$ and
  $(r(\Ac) - \mu)^ku = 0$, then $u$ must be in $\dom
  (\Ac^\infty)$. Having thus shown that
  $E_\mu^\infty(r(\Ac)) \supseteq E_\mu(r(\Ac))$, the proof
  of~\eqref{eq:9} is complete.

  It now suffices to prove~\eqref{eq:13} with $E_\mu(r(\Ac))$ replaced
  by $E_\mu^\infty(r(\Ac))$. Using polynomials 
  \[
    q(z) = \prod_{j=1}^N (z_j - z)^{\nu_j}, \qquad
    p(z) = \sum_{j=1}^N \sum_{i=1}^{\nu_j}
    \og_{j,i} \prod_{k \ne j} (z_k - z)^{\nu_k - i},
  \]
  we may rewrite the rational function as $r(z) = p(z) / q(z)$. Then
  \begin{equation}
    \label{eq:15}
    r(z) - \mu = \frac{p(z) - \mu \,q(z)}{ q(z)}.
  \end{equation}
  Factorizing the numerator, we obtain finitely many distinct complex
  numbers $\lambda_\ell$, $\ell=1,\dots, L,$
  and associated integers $n_\ell$ such that
  \begin{equation}
    \label{eq:15.1}
    p(z) - \mu \,q(z) = 
    \alpha \prod_{\ell=1}^L    (\lambda_\ell - z)^{n_\ell}
\end{equation}
  for some scaling factor $\alpha$. 

  Let us continue supposing that $\{ \lambda_1, \dots \lambda_L\}$
  does not intersect $\{ z_1, \dots, z_N\}$.  (The intersecting case
  is dealt with at the end.) Then, from \eqref{eq:15} and~\eqref{eq:15.1},
  it is clear
  that $r(z) - \mu=0$ if and only if $z = \lambda_\ell$ for some
  $\ell$. In other words, the inverse image of $\mu$ equals 
  \[
    r^{-1}\{ \mu \} = \{ \lambda_1, \dots \lambda_L\}.
  \]
  Let $u \in \dom (\Ac^\infty)$.  Since the factors of $r(\Ac) - \mu$
  commute in $\dom (\Ac^\infty)$,
  \[
    (r(\Ac) - \mu)^k u = q (\Ac)^{-k} \prod_\ell (\lambda_\ell - \Ac)^{n_\ell k} u. 
  \]
  The factor $q (\Ac)^{-k}$ is a composition of powers of injective
  resolvent operators since $z_j \in \rho(\Ac)$. Hence, any
  $u \in \dom (\Ac^\infty)$ satisfies $(r(\Ac) - \mu)^k u =0$ if and
  only if
  \[
    \prod_\ell (\lambda_\ell - \Ac)^{n_\ell k} u = 0,
  \]
  which, by the argument of~\cite{Yam1971-1} leading
  to~\eqref{eq:Yamamoto-decomposition}, is equivalent to
  \[
    u \in \bigoplus_{\ell =1}^L \ker( \lambda_\ell - \Ac)^{n_\ell k}.
  \]
  This proves that
  \begin{equation}
    \label{eq:16}
        E_\mu^\infty(r(\Ac)) =
    \bigoplus_{\lambda \,\in\, r^{-1} \{ \mu\} } E_\lambda^\infty (\Ac)      
  \end{equation}
  holds for the current case.

  To finish the proof, consider the postponed case where there
  are integers 
  $\ell$ and $i$ such that $\lambda_\ell = z_i$. If $n_\ell \le  \nu_i$, then we may 
  cancel off the common factor $(z_i - z)^{n_\ell}$
  from the numerator and denominator of~\eqref{eq:15} and reuse the above arguments with
  $\tilde q(z) = q(z)/ (z_i - z)^{n_\ell}$, a polynomial of lower degree,
  in place of $q(z)$.
  Then,   $r^{-1} \{ \mu \}$ no longer contains the associated~$\lambda_\ell$,
  so  $\lambda_\ell$ no longer 
  contributes to the sum of~\eqref{eq:16}. If $n_\ell>\nu_i$, then the
  numerator in~\eqref{eq:15} will still contain a factor of
  $\lambda_\ell -z$, so $r^{-1}\{ \mu\}$ continues to contain
  $\lambda_\ell$. But $E_{\lambda_\ell}(\Ac)$ is trivial, since
  $\lambda_\ell=z_i$ is in the resolvent
  set of $\Ac$. Hence it only contributes the trivial space to the sum
  in~\eqref{eq:16} and can be removed from the sum.
\end{proof}

\begin{example}[A finite-dimensional case] \label{ex:3x3}
  Consider a linear operator  $\Ac$ on $\Hc = \CC^3$  
  represented by the matrix 
  \[
    \Ac =
    \begin{bmatrix*}[r]
      -\frac{1}{2} & 0 & 0 \\
      0 & \frac{1}{2} & 1 \\
      0 & 0 & \frac{1}{2}
    \end{bmatrix*}
  \]
  in Jordan canonical form.  Here, abusing notation, we have
  identified the linear operator with its matrix representation
  obtained using  the standard unit basis vectors denoted by
  $e_1, e_2,$ and $e_3$.
  The eigenvalues of \(\Ac\) are \(-\frac{1}{2}\) and
  \(\frac{1}{2}\), with algebraic multiplicities \(1\) and \(2\),
  respectively.  Their generalized eigenspaces are
  $E^\infty_{-1/2} = \spn(e_1)$ and
  $E^\infty_{-1/2} = \spn(e_2, e_3)$.  Define the rational function
  \begin{equation}
    \label{eq:feast-N2}
    r(z) = -\frac{1}{z^2 + 1}
    = \frac{\ii}{2(z - \ii)} + \frac{-\ii}{2(z + \ii)}.    
  \end{equation}
  It is immediately verified that the image of all eigenvalues of
  $\Ac$ under $r$ equal \(-\frac{4}{5}\).  In particular, the inverse
  image $r^{-1}_\mu$ of \(\mu = -\frac{4}{5}\) equals 
  \begin{equation}
    \label{eq:3x3r-inv}
    r^{-1}_{-4/5} = \left\{\frac{1}{2}, -\frac{1}{2}\right\},
  \end{equation}
  the entire spectrum of $\Ac$.
  Functions of Jordan forms are easy to compute and in this case we obtain 
  \[
    r(\Ac) = 
    \begin{bmatrix*}[r]
      -\frac{4}{5} & 0 & 0 \\
      0 & -\frac{4}{5} & \frac{16}{25} \\[.25em]
      0 & 0 & -\frac{4}{5}
    \end{bmatrix*}.
  \]
  In particular, it shows that the generalized eigenspace of the
  eigenvalue $-4/5$ is the entire space $\CC^3$. This eigenspace also
  admits the direct decomposition
  \begin{align*}
    E_{-4/5}(r(\Ac)) = 
    \CC^3
    & =
    \spn(e_1)
    \oplus
    \spn(e_2, e_3)
    \\
    & =
    E_{-1/2}^\infty(\Ac) \oplus E_{1/2}^\infty(\Ac),
  \end{align*}
  thus illustrating the identity of
  Lemma~\ref{lem:rational-gen-eigenspaces} in a simple
  finite-dimensional case.
\end{example}

\begin{example}[The Cayley transform of an operator]
  \label{ex:cayley}
  Consider the Cayley transform
  \begin{equation}
    \label{eq:mobius}
    r(z) = \frac{z - \ii}{z + \ii} = 1 - \frac{2\ii}{z + \ii},
  \end{equation}
  which maps the upper half-plane to the unit disk~\cite[Theorem
  6.3.6]{GreeKran06}.  This is a M\"obius transform (or a linear
  fractional transformation), see
  e.g.,~\cite[Section~III.3]{Conway2012} or
  \cite[Section~6.3]{GreeKran06}.  In this example, we set the linear
  operator \(\Ac\) on \(\Hc = \CC^3\) as represented by the matrix in
  Jordan canonical form
  \[
    \Ac =
    \begin{bmatrix*}[r]
      10 - \ii & 0 & 0 \\
      0 & 10 + \ii & 1 \\
      0 & 0 & 10 + \ii
    \end{bmatrix*},
  \]
  whose eigenvalues are \(10 - \ii\) and \(10 + \ii\), with algebraic
  multiplicities \(1\) and \(2\), respectively.
  The image of these eigenvalues under the Cayley transform
  \(r\) defined in~\eqref{eq:mobius} are
  \(r(10 - \ii) = \frac{5 - \ii}{5}\) and
  \(r(10 + \ii) = \frac{25 - 5\ii}{26}\).
  As M\"obius transforms are invertible, the inverse images
  \(r^{-1}_\mu\), for \(\mu_1 = \frac{5 - \ii}{5}\) and
  \(\mu_2 = \frac{25 - 5\ii}{26}\) are singletons.
  The  Cayley transform of \(\Ac\) is 
  \[
    r(\Ac) =
    \begin{bmatrix*}[r]
      \mu_1 & 0 & 0 \\
      0 & \mu_2 & \frac{5 + 507\ii}{676} \\
      0 & 0 & \mu_2
    \end{bmatrix*},
  \]
  whose eigenvalues are \(\mu_1\) and \(\mu_2\), with algebraic
  multiplicities \(1\) and \(2\), respectively.
  It can be checked readily that the generalized eigenspaces of
  \(r(\Ac)\) satisfy the identity of
  Lemma~\ref{lem:rational-gen-eigenspaces}.

  It is worth noting that \(\abs{\mu_1} = \sqrt{26}/5 > 1\) and
  \(\abs{\mu_2} = 5/\sqrt{26} < 1\), which is consistent with the
  mapping properties of the Cayley transform.  This example
  illustrates how M\"obius transforms can be used to map some parts of
  the spectrum of an operator into dominant eigenvalues (larger in
  absolute value than the rest of the spectrum) of the transformed operator.
\end{example}

\begin{remark}[General rational functions]
  A general rational function is a quotient of two polynomials $p(z)$
  and $q(z)$ over the complex field, namely, \(r(z) = p(z)/ q(z)\).
  It is well-known (see, e.g.,~\cite[Eq.~V.1.9]{Conway2012} or
  \cite[Theorems~IV.5.2--IV.5.3]{Lang2012}) that we can express \(r\)
  in a partial fraction expansion: denoting the zeros of $q(z)$ by
  $z_i$, with multiplicity $\nu_i$, there exist coefficients
  \(\omega_{j, i}\in \mathbb{C}\) and a polynomial
  \(\hcal(z)\) representing the holomorphic part of \(r\),
  such that
  \begin{equation}
    \label{eq:general-r}
    r(z) = \hcal(z) + \scal(z), \quad \text{ where } \quad \scal(z) 
    =    \sum_{j=1}^\nquad 
    \sum_{i=1}^{\nu_j}\frac{\omega_{j, i}}{(z_j - z)^i}.
  \end{equation}
  As before, assume that $z_j \in \rho(\Ac)$. Then, the proof of
  Lemma~\ref{lem:rational-gen-eigenspaces} can be generalized to this
  case to show that  {\em for any $\mu \in \CC$
    such that
    \begin{equation}
      \label{eq:4}
      \hcal_\mu^{-1} \cap \sigma(\Ac) = \emptyset, 
    \end{equation}
    the equality~\eqref{eq:13} continues to hold} for the $r(z)$
  in~\eqref{eq:general-r}.  Here
  $\hcal_\mu^{-1}:= \{ z \in \CC: \hcal(z) = \mu \}$ denotes the set
  of zeros of \(\hcal(z) -\mu\).  The change in the proof needed to achieve this 
  generalization involves replacing~\eqref{eq:11} by
  \[
    u = (\hcal(\Ac) - \mu)^{-1} \left( v  - \scal(\Ac) u \right),
  \]
  where the inverse exists and can be written as a composition of
  resolvents due to~\eqref{eq:4} and the fundamental theorem of algebra.
  This proves that $u$ is in
  \(\dom(\Ac)\) whenever \(v\) is in \(\dom(\Ac)\), and iterating, one
  establishes~\eqref{eq:10}, \eqref{eq:9}, and the remainder of proof
  of the lemma as before.  Note that condition~\eqref{eq:4}
  generalizes the condition $\mu \ne \og_0$ of
  Lemma~\ref{lem:rational-gen-eigenspaces}: indeed, when
  $\hcal(z) = \og_0$, a polynomial of degree zero, $\mu \ne \og_0$
  implies that $\hcal^{-1}_\mu = \emptyset$ so~\eqref{eq:4} holds.
\end{remark}

\section{Estimating the error in eigenspace}
\label{sec:eigenspace-error-estimate}

Often, in practical problems, the need arises to numerically
approximate the eigenspace corresponding to a finite cluster of
eigenvalues in $\sigma(\Ac)$.
Here, ``cluster''  refers to a finite set of eigenvalues that are
relatively close to each other and well separated from the rest of
the spectrum.
The main result of this section, Theorem~\ref{thm:espace-gap-estimator}, provides a computable reliable eigenspace error estimator  under suitable assumptions.

Recall that  an eigenvalue is called an
isolated eigenvalue if it has an open neighborhood in $\CC$ whose all
other points belong to $\rho(\Ac)$.  Let $\vL$ denote a finite set of
isolated eigenvalues of $\Ac$.
Let $m(\lambda, \Ac)$ denote the
algebraic multiplicity of $\lambda$ as an eigenvalue of
$\Ac$. For all  $z \in \rho(\Ac)$, we set 
$m(z, \Ac)=0$.
Henceforth, we assume that $m(\lambda, \Ac) < \infty$ for each
eigenvalue $\lambda$ in the cluster~$\vL$.

We are interested in controlling the error in approximation of the
algebraic eigenspace  $E$ of the cluster of eigenvalues
$\vL$.  Let the set $\vL$ be enclosed within $\vG$, a positively
oriented, bounded, simple, closed contour that lies entirely in
$\rho(\Ac)$ and encloses no element of $\sigma(\Ac)$ other than those in $\vL$.
Then the Riesz projector associated to the cluster $\vL$ is
\begin{equation}
  S = \frac{1}{2\pi \ii} \oint_\vG
  R \parens{z} \diff{z}.
  \label{eq:feast-spectral}
\end{equation}
Hence the wanted eigenspace is the range of $S$, 
\[
  E = \ran S \subset \dom (\Ac),
\]
and 
\[
  \dim E = \sum_{\lambda \in \vL} m(\lambda, \Ac) < \infty.
\]

Before approximating $E$, we map the cluster using rational functions of the
form~\eqref{eq:feast-filter}, restricted further, for simplicity, to the form
\begin{equation}
    r(z) =  \og_0 +  \sum_{j=1}^\nquad 
    \frac{\omega_{j}}{z_j - z}
    \label{eq:feast-filter-2}
\end{equation}
for some $\og_j \in \CC$. To mimic the behavior of the operator-valued
integral~\eqref{eq:feast-spectral}, it usually suffices to consider
rational functions of the restricted form~\eqref{eq:feast-filter-2}
(as will be clear from Example~\ref{ex:Butterworth} later).
The set $\vL$ is mapped by $r$ to 
\begin{equation}
  \label{eq:mapped-ews}
  \vU = \{ r(\lambda): \lambda \in \vL\}.
\end{equation}
Let
\begin{equation}
\label{eq:Sr-defn}
  S_r = r(\Ac) = \og_0 + \sum_{j=1}^\nquad \omega_j R(z_j).
\end{equation}
Many algorithmic techniques construct specific rational functions $r$
of interest that make $\vU$ more easily capturable than $\vL$, e.g.,
rational functions $r$ that make $\vU$ into a set of ``dominant
eigenvalues'' of $S_r$ easily captured by power iterations.
Postponing discussion of such details,
we proceed making the following assumption.

\begin{assumption}[On Rational Functions]
  \label{asm:r-asm}
  Assume that $r$ is of the form~\eqref{eq:feast-filter-2}, 
  \begin{equation}
    \label{eq:r-asm}
    \og_0 \notin \vU \quad\text { and } \quad
    \bigcup_{\mu \in \vU} r_\mu^{-1}\cap \sigma(\Ac) \,\subseteq\, \vL
  \end{equation}
  where $r_\mu^{-1}$ is defined in \eqref{eq:r-mu-inv}
  and $\vU$ is as in~\eqref{eq:mapped-ews}.
\end{assumption}

Let $\vT$ be a simple closed
curve in the complex plane enclosing $\vU$ and not enclosing the
origin nor any point of $\sigma(S_r) \setminus \vU$. Define
\begin{align*}
  P_r & = \frac{1}{2\pi \ii} \oint_\vT( z - S_r)^{-1}\, \dz,
  \\
  E_r & = \ran P_r.
\end{align*}

\begin{lemma}
  \label{lem:E=Er}
  Suppose Assumption~\ref{asm:r-asm} holds. Then every $\mu \in \vU$ is an
  eigenvalue of $S_r$ of algebraic multiplicity
  \begin{equation}
    \label{eq:13-dim}
    m(\mu, S_r) = \sum_{\lambda \,\in\, r^{-1}_\mu}   m(\lambda, \Ac).
  \end{equation}
  Moreover, the spaces  $E_r$ and $E$ coincide.
\end{lemma}
\begin{proof}
  Let $\mu \in \vU$.  By~\eqref{eq:r-asm}, $\mu \ne \og_0$.  Hence,
  Lemma~\ref{lem:rational-gen-eigenspaces} applies and \eqref{eq:13}
  holds.    By counting dimensions on both sides
  of~\eqref{eq:13}, the stated equality~\eqref{eq:13-dim} follows.

  Clearly, $\mu$ is an eigenvalue of $S_r$ if and only if its
  algebraic eigenspace $E_\mu(S_r)$ is nontrivial.  The
  equality~\eqref{eq:13} implies that the eigenspace $E_\mu(S_r)$ is
  nontrivial if and only if there is a $\lambda$ in $r_\mu^{-1}$ such
  that $E_\lambda^\infty(\Ac)$ is nontrivial, i.e., if and only if
  there is a $\lambda$ in $r_\mu^{-1} \cap \sigma(\Ac).$ Since
  $r_\mu^{-1} \cap \sigma(\Ac)$ is contained in $\vL$
  by~\eqref{eq:r-asm} of Assumption~\ref{asm:r-asm}, the
  equality~\eqref{eq:13} implies the sum of the eigenspaces of all
  $\mu$ in $\vU$ satisfies   
  \[
    \sum_{\mu \in \vU}
    E_\mu(r(\Ac)) =
    \sum_{\mu \in \vU}\;
    \bigoplus_{\lambda \,\in\, r^{-1}_\mu \cap \sigma(\Ac) }
    E_\lambda^\infty (\Ac)
    =
    \sum_{\lambda \in \vL} E_\lambda^\infty (\Ac).
  \]
  The left and right hand sides equal $E_r$ and $E$, respectively, so $E_r = E$.
\end{proof}

\begin{example}
  Reviewing the case of the $r(z)$ and $\Ac$ of Example~\ref{ex:3x3},
  putting $\vL = \{-1/2, 1/2\}$, we have $\vU = \{-4/5\}$ (which
  clearly does not contain $\og_0=0$).  Then~\eqref{eq:3x3r-inv} shows
  that Assumption~\ref{asm:r-asm} holds.
\end{example}

\begin{example}[The Cayley transform revisited]
  Reviewing the case of the Cayley transform $r(z)$ and $\Ac$ of
  Example~\ref{ex:cayley}, putting $\vL = \{10 - \ii, 10 + \ii\}$, we have
  $\vU = \{ (5 - \ii)/5, (25 - 5\ii)/26\}$. Clearly $\og_0 = 1$ for this $r(z)$, which is not in contained in $\vU$, so the  first condition
  in~\eqref{eq:r-asm} holds.  The second
  condition in~\eqref{eq:r-asm} also holds since the inverse images
  of the two elements of $\vU$ are exactly the two eigenvalues in
  $\vL$.
\end{example}

\begin{example}[Rational function in inverse iterations]
  When a single isolated eigenvalue $\lambda_1 \in \CC$ of a linear
  operator $\Ac$ on $\Hc$ is to be targeted, the cluster
  $\vL =\{ \lambda_1\}$ is a singleton.  A standard approach in this
  case is to perform an inverse iteration to capture the eigenspace of
  $\lambda_1$. This iteration consists of repeated application of
  powers of $r(\Ac)$ on some initial subspace for an $r$ of the form
  \[
    r(z) = \frac{1}{ z_1 -z},
  \]
  using a ``guess'' $z_1 \in \CC$ that is close but not equal to
  $\lambda_1$.  In this case the nonzero complex number
  $\mu = r(\lambda_1) = (z_1 - \lambda_1)^{-1}$ is the only element of
  $\vU$.  We also immediately see that $\mu = r(\lambda)$ if and only
  if $\lambda = \lambda_1$, i.e., $r_\mu^{-1} = \vL$.  Clearly
  Assumption~\ref{asm:r-asm} holds in this case.
\end{example}

\begin{example}[A rational function in FEAST
  iterations] \label{ex:Butterworth} The so-called ``FEAST'' contour integral
  eigensolver~\cite{Pol2009-1,GopGruOva2020-1} uses rational functions
  obtained by replacing the contour integral over $\vG$
  in~\eqref{eq:feast-spectral} by a quadrature. Consider the case of a
  circular contour $\vG$ centered at some $O \in \CC$ of radius $R>0$,
  parameterized by $O + R \phi\,e^{\ii \theta}$ with some initial
  phase factor $\phi$ of unit magnitude.
  Rewriting the contour integral over $\vG$ as an integral over $[0, 2\pi]$ 
  based on this parameterization,  and applying
  the trapezoidal rule with $N$ points, we obtain a rational function of
  the form~\eqref{eq:feast-filter} with
  \begin{equation}
    \label{eq:Butterworth-z-w}
    \og_j = \hat\og_N^{j-1} R \phi/N, \qquad z_j = R \phi\,\hat\og_N^{j-1} + O
  \end{equation}
  where $\hat\og_N = e^{\ii 2 \pi/N}$. A commonly made choice is
  $\phi = \pm e^{\ii \pi/N}.$ (The case  $N=2$ with $O=0$ and $R=1$ yields
  \eqref{eq:feast-N2} in Example~\ref{ex:3x3}.)  
  Such rational functions are often known
  as ``Butterworth filters'' in signal processing~\cite{Ham1998-1}.
  Note that the powers of $\hat \og_N$ appearing above, namely
  $\hat\og_N^j$ for $j = 0, \dots, N-1$, are the $N$th roots of
  unity. By a partial fraction expansion (see
  \cite[Example~2.2]{GopGruOva2020-1}) one can show that with the
  settings in~\eqref{eq:Butterworth-z-w}, we have
  \begin{equation}
    \label{eq:rz-two-forms}
    r(z) =  \sum_{j=1}^\nquad 
    \frac{\omega_j}{z_j - z} =
    \left[1 - \left( \frac{z - O}{R\phi} \right)^N   \right]^{-1}.
  \end{equation}
  Let $D = \{ z \in \CC: | z - O | < R \}$ be the disk enclosed by
  $\vG$. Now suppose we are in the setting described in the beginning
  of this section, where we have an eigenvalue cluster of interest
  $\vL$ that lies in $D$ and
  \begin{equation}
    \label{eq:4intersect}
    \sigma(\Ac) \cap D = \vL. 
  \end{equation}
  Then we claim that Assumption~\ref{asm:r-asm} holds. Indeed, since
  $|z - O| < |R \phi|$, the second expression for $r(z)$ in
  \eqref{eq:rz-two-forms} and triangle inequality yields
  \[
    |r(z)| > \frac 1 2 \qquad \text{ for all } z \in D.
  \]
  Hence none of the elements of  $\vU$ can equal $\og_0 =0$. Next, consider 
  any $\mu \in \vU$. There must be a $\lambda \in \vL$
  such that $\mu = r(\lambda)$. Any $z $ in $r_{\mu}^{-1}$ must therefore 
  satisfy $r(z) = r(\lambda)$, or equivalently
  $(z - O)^N = (\lambda - O)^N$. Hence $z$ must equal one of the
  complex numbers
  \[
    \zeta_\ell = O +  (\lambda - O)\, \hat\og_N^\ell,
    \qquad
    \ell=0, \ldots, N-1,
  \]
  all of which lie in $D$, i.e., we have shown that
  $r_\mu^{-1} = \{ \zeta_\ell: \ell=0, \ldots, N-1\} \subset
  D$. By~\eqref{eq:4intersect}, the only elements of the spectrum of
  $\Ac$ in $D$ are those in $\vL$, so  it follows that
  $r_\mu^{-1} \cap \sigma(\Ac) \subseteq \vL$, thus
  verifying~\eqref{eq:r-asm}.
\end{example}

\bigskip

Next, we proceed to approximate $S_r$ and study how the eigenspaces
change under such approximations. Throughout we suppose that there is a
Hilbert space $V$, normed by $\| \cdot \|_{V}$, that is continuously
embedded into $\Hc$ such that $\dom(\Ac ) \subseteq V$, $E \subset V$, and that $V$
is an invariant subspace of the resolvent operator $R(z)$ for all
$z \in \rho(\Ac)$. Typical
examples of $V$ include the whole space $\Hc,$ or the set $\dom (\Ac)$
after making into a Hilbert space using the graph norm of $\Ac$, or
the domain of the a sesquilinear form from which $\Ac$ arises. Such
examples were discussed in~\cite{GopGruOva2020-1}, where the next
assumption can also be found. 
The operator norm induced by the norm on~$V$ is also denoted by
$\| \cdot \|_V$, e.g., since $V$ is an invariant subspace of the resolvent
$R(z)$, it can be considered as an operator on $V$ whose operator norm is 
\[
  \| R(z) \|_V = \sup_{0 \ne v \in V} \frac{ \| R(z) v\|_V }{ \| v \|_V }.
\]
This operator  norm appears in the next assumption.

\begin{assumption}[On Resolvent Approximations]
  \label{asm:resolvent-approx}
  We assume that there are
  finite-dimensional subspaces $V_h \subset V$ (indexed by a
  discretization parameter $h$ approaching zero) and finite-rank
  operators $R_h(z) : \Hc \to V_h$ such that
  \begin{equation}
    \label{eq:1}
    \lim_{h \to 0} \| R_h(z_k) - R(z_k) \|_V = 0
  \end{equation}
  for every $k = 1, \dots, N.$
\end{assumption}

Under this assumption, the operator   
\begin{equation}
  \label{eq:Srh-defn}
  \Srh = \og_0 + \sum_{j=1}^\nquad \omega_j R_h(z_j)
\end{equation}
is an approximation of $S_r$. Note that
Assumption~\ref{asm:resolvent-approx}, in particular, restricts us to
consider only operators $\Ac$ with compact resolvent (since $V_h$ is
finite dimensional).  Equation~\eqref{eq:1} allows us to define a
natural approximate eigenspace as follows. It 
implies that $\|\Srh - S_r\|_V$ converges to
$0$ as $h$ approaches $0$. Hence, given any open disc enclosing a single 
isolated eigenvalue $\mu \in \vU$ of $S_r$, for sufficiently small
$h$, there are exactly as many eigenvalues of $\Srh$ in the same disc
as the algebraic multiplicity $m(\mu, S_r).$ In particular, this
implies that, for sufficiently small $h$, the contour $\vT$ is in the
resolvent set of $\Srh$ and encloses as many eigenvalues of joint
multiplicity equal to $\dim E_r.$ Hence, the integral
\[
  \Prh  = \frac{1}{ 2 \pi \ii} 
  \oint_\Theta \left(z - \Srh \right)^{-1} \,\dz
\]
is well defined and its range 
\begin{equation}
  \label{eq:Eh-defn}
    E_h = \ran \Prh 
\end{equation}
has the same dimension as the space $E_r$ it intends to approximate: 
\begin{equation}
  \label{eq:2}
  \dim E_h = \dim E_r.
\end{equation}
We tacitly assume throughout that {\em $h$ has been made small enough
for~\eqref{eq:2} to hold and for the above definitions to make sense.}
In practice, one computes (a basis for) the approximate eigenspace $E_h$
in~\eqref{eq:Eh-defn} through subspace iterations---see, e.g., \cite[Theorem~3.4]{GopGruOva2020-1} which shows how certain subspace iterates approach $E_h$,
or see 
Section~\ref{sec:numerical-results}, where we use 
a subspace iteration based on operator-valued contour integrals to compute~$E_h$.

Comparison of the spaces $E$ and $E_h$ is done through the gap metric in $V$. To
define it, first let
\[
  \delta(U, W) =
  \sup_{u \in U,\; \|u \|_V=1 } \dist_V(u, W)  = \| (I - Q_W) Q_U \|_V
\]
for any two closed subspaces \(U\) and \(W\) of $V$. Here
\(\dist_V(u, V) = \inf_{v \in V} \norm{u - v}_V\),
$I$ denotes the identity on $V$,
and $Q_W$ denotes the $V$-orthogonal projection into $W$.
Symmetrizing this, we define the gap by
\[
  \gap_V(U, W) 
  = \max \Big( \delta(U, W), \;\delta(W, U) \Big).
\]
The argument of the next lemma is similar to the proof of
\cite[Theorem~4.1]{GopGruOva2020-1}, but the key difference is that
the supremum appearing in Lemma~\ref{lem:gap-bound} is taken over the
approximating space $E_h$, not the exact eigenspace~$E$.

\begin{lemma}
  \label{lem:gap-bound}
  If Assumption~\ref{asm:resolvent-approx} holds, 
  then there exist $C_r>0$ and $h_0>0$ such
  that for all $h < h_0$,
  \begin{equation}
    \label{eq:gapEEh}
    \gap_V( E_r, E_h) \le C_r
    \sup_{ e_h \in E_h, \; \| e_h \|_V = 1 }
    \left\| ( S_r - \Srh) e_h \right\|_V.
  \end{equation}
\end{lemma}
\begin{proof}
  We begin by noting that 
  \begin{align*}
    P_r - \Prh 
    & =
      \frac{1}{2\pi \ii} \oint_\vT
    \Big[ (z - S_r)^{-1} - (z - \Srh)^{-1}   \Big] \, \dz
    \\
    & =
      \frac{1}{2\pi \ii} \oint_\vT
      (z - S_r)^{-1}
      \big[  S_r - \Srh   \big]  (z - \Srh)^{-1} \, \dz.
  \end{align*}
  Since generalized eigenspaces of any operator are invariant
  subspaces of its resolvent,
  \[
    (z - \Srh)^{-1} E_h \subseteq E_h.
  \]
  Consequently,   for any $e_h \in E_h$, we  have
  \begin{align*}
    \| (\Prh& - P_r) e_h \|_V
     =      
      \frac{1}{2\pi}
      \left\|
      \oint_\vT
      (z - S_r)^{-1}
      (S_r - \Srh  ) Q_{E_h}  (z - \Srh)^{-1} e_h \, \dz
      \right\|
    \\
    & \le      
      \left(\frac{1}{2\pi}
      \oint_\vT
      \|(z - S_r)^{-1}\|_V
      \left\| (z - \Srh)^{-1} \right\|_V \,
      \dz
      \right)
      \| (S_r - \Srh  )Q_{E_h} \|_V  \| e_h\|_V
  \end{align*}
  Bound the term within parentheses above by $C_r$ and note that we
  may choose such a bound independently of $h$ when $h$ is
  sufficiently small due to~\eqref{eq:1} of Assumption~\ref{asm:resolvent-approx}.
  We use the above estimate
  to conclude that
  \begin{align}
    \nonumber 
    \delta(E_h, E_r)
    & = \sup_{e_h \in E_h, \; \| e_h \|_V = 1 } \dist_V(e_h, E_r)
    \\ \nonumber 
    & \le  \sup_{e_h \in E_h, \; \| e_h \|_V = 1 }\| e_h - P_r e_h \|_V
    \\ \nonumber 
    & =  \sup_{e_h \in E_h, \; \| e_h \|_V = 1 }\| (\Prh - P_r) e_h \|_V
    \\ \label{eq:6}
    & \le C_r \| (S_r - \Srh  ) Q_{E_h} \|_V.
  \end{align}
  This implies, by~\eqref{eq:1}, that there is an $h_0>0$ such that
  $\delta(E_h, E_r) = \| (I - Q_{E_r} ) Q_{E_h}\|_V < 1$ for all
  $h < h_0$. Then the two alternatives
  of~\cite[Theorem~I.6.34]{Kat2013-1} apply: they imply that there is
  a subspace $\tilde E_r \subseteq E_r$ such that
  \begin{equation}
    \label{eq:3}
    \delta(E_h, E_r) = \delta(E_h, \tilde E_r) = \gap(E_h, \tilde E_r) < 1.
  \end{equation}
  Note that
  \begin{align*}
    \dim \tilde{E}_r
    & = \dim E_h
    && \text{ by~\eqref{eq:3}}
    \\
    & = \dim E_r
    && \text{ by~\eqref{eq:2}}.
  \end{align*}
  Since $\tilde E_r \subseteq E_r$, this implies that
  $E_r = \tilde E_r$. Hence $\gap(E_h, \tilde E_r) = \gap(E_h, E_r)$,
  which by~\eqref{eq:3} also equals $ \delta(E_h, E_r)$. Therefore,
  \eqref{eq:6} shows that
  $ \gap(E_h, \tilde E_r) =
  \gap(E_h, E_r) =  \delta(E_h, E_r)\le C_r \| (S_r - \Srh)
  Q_{E_h} \|_V$, proving the result.
\end{proof}

\bigskip

The lemma now leads to our main result, which produces an error
estimator for the gap between the true and the approximate eigenspace
by suitably combining error estimators for certain linear source
problems.  Suppose there is a Hilbert space $Y$, with
norm~$\| \cdot \|_{Y}$, into which such error estimators for source
problems can be computed---they are denoted by $\EE$ in the next
assumption.

\begin{assumption}[On Source Error Estimators]
  \label{asm:source-ee}
  There is a $C_0>0$ and a linear
  operator $\EE: V_h \to Y$ (whose application to any $v \in V_h$ is easily computable) satisfying
  \begin{equation}
    \label{eq:EE}
    \| S_r v_h - \Srh v_h \|_V \le C_0 \| \EE v_h\|_{Y}, \qquad v_h \in V_h.
  \end{equation}
\end{assumption}

To understand the relevance of this assumption, subtract~\eqref{eq:Srh-defn} from~\eqref{eq:Sr-defn} to get 
\[
  S_r v_h - \Srh v_h =
  \sum_{j=1}^\nquad \omega_j \big( R(z_j) v_h - R_h(z_j) v_h\big).
\]
Each summand above can be estimated by observing that
$u_h = R_h(z_j) v_h$ is an approximation to the solution $u$ of the
source problem $ (z - \Ac) u = v_h$. This is why we referred to
$\EE$ as an error estimator from source problems.

Our goal is to repurpose such source problem error estimators to
estimate the error in the eigenspace approximation~$E_h$.  Assuming
that a basis $ e_h^i$, $i=1, \dots, L, $ for the eigenspace
approximation $E_h$ has been computed, we define the $L \times L$
matrices
\begin{equation}
\label{eq:5}
G_{ij} = ( \EE e_h^j, \EE e_h^i)_Y, \qquad
M_{ij} = (e_h^j, e_h^i)_V
\end{equation}
using the inner products of $Y$ and $V$, respectively.
Consider the small $L \times L$ Hermitian generalized eigenproblem
$G x = \lambda Mx$.
Let its maximal eigenvalue be denoted by
$\hat \lambda$, i.e.,
\[
  \hat\lambda = \max \sigma(M^{-1}G),
\]
and let
$\hat x$ denote an eigenvector of eigenvalue $\hat \lambda$.
Let 
\begin{equation}
\label{eq:ehmax}
\tilde e_h = \sum_{i=1}^L \hat x_i e_h^i, \qquad
\hat e_h = \frac{\tilde e_h } { \| \tilde e_h \|_V}.
\end{equation}
The latter function in $E_h$ serves as an eigenspace error indicator
after an application of $\EE$, whose global reliability is proved 
next. The proof also makes it clear that $\hat e_h$ does not depend on the
specific choice of $e_h^j$ but only on the space $E_h$.

\begin{theorem}[Eigenspace Gap Estimator]
  \label{thm:espace-gap-estimator}
  Suppose Assumptions~\ref{asm:r-asm},~\ref{asm:resolvent-approx}
  and~\ref{asm:source-ee} hold.  Then, there is an $h_0>0$ and $C>0$
  such that for all $h < h_0$,
  \[
    \gap_V(E_h, E) \le C \| \EE \hat e_h\|_Y
  \]
  where the estimator $\EE \hat e_h$ is computed using (only) the approximate eigenspace~$E_h$.
\end{theorem}
\begin{proof}
  Since Assumption~\ref{asm:r-asm} holds, by Lemma~\ref{lem:E=Er}, we
  have $E = E_r$. Hence verifying the conditions of
  Lemma~\ref{lem:gap-bound} (by virtue of the remaining two
  assumptions), its estimate~\eqref{eq:gapEEh} implies
  \begin{align}
    \nonumber 
    \gap_V( E, E_h)
    & = \gap_V( E_r, E_h) \le C_r
    \sup_{ e_h \in E_h, \; \| e_h \|_V = 1 }
      \left\| ( S_r - \Srh) e_h \right\|_V
    \\ \label{eq:supEE}
    &
      \le C_rC_0 \sup_{ e_h \in E_h}
      \frac{\left\| \EE e_h \right\|_Y}{\| e_h \|_V}.
  \end{align}
  Writing any $e_h \in E_h$ as $e_h = \sum_{i=1}^L x_i e_h^i$, we find that
  \begin{align*}
    \frac{\left\| \EE e_h \right\|_Y^2}{\| e_h \|_V^2}
    = \frac{x^* G x}{ x^*M x}, 
  \end{align*}
  which is the Rayleigh quotient of the eigenproblem
  $G x = \lambda Mx$. It is maximized by an eigenvector $\hat x$ of
  its maximal eigenvalue $\hat \lambda$. Hence the supremum in
  \eqref{eq:supEE} is attained by $\hat e_h$ of~\eqref{eq:ehmax} and
  the result follows with $C = C_rC_0$.
\end{proof}

We conclude this section with a discussion of
Theorem~\ref{thm:espace-gap-estimator}, a  result  useful for
adaptive algorithms for eigenvalue clusters.  As seen
from the above proof, the computable source estimator $\EE \hat e_h$ depends
only on the computed eigenspace $E_h$, and not on individual
eigenvectors within it. In this sense, our estimator is cluster
robust.   Computing a supremum over a finite dimensional
approximate space for error estimation, as used in the above proof,
is a technique that can also be found in other works, e.g.,
in~\cite{LiuVej2022-1}, where the authors develop estimates of gap between exact and
approximate eigenspaces of a generalized selfadjoint eigenvalue
problem. The idea of using source problem error estimates for
eigenvalue adaptivity can also be found in the works
of~\cite{GruOva2009-1, BanGruOva2013-1,%
  GianiEtAl2016-1, GiaGruHakOva2021-1}, where residuals generated by
solution operators, applied to approximate eigenfunctions, feature
prominently.  Such residuals are then further estimated by using an
auxiliary space (e.g., appropiate bubble spaces in the finite element
context), which then provide control of relative eigenvalue errors.
Specific new examples of locally computable $\EE$ for a selfadjoint problem appear in the
next section (Section~\ref{sec:application}) and a nonselfadjoint example can be found in Section~\ref{sec:numerical-results}.



\section{Application to finite element discretizations} \label{sec:application}

In this section, we apply the framework for eigenspace error
estimators we developed in Section~\ref{sec:eigenspace-error-estimate},
specifically Theorem~\ref{thm:espace-gap-estimator}, to some finite element examples. We illustrate how to verify
Assumptions~\ref{asm:resolvent-approx} and~\ref{asm:source-ee} for a
typical partial differential operator. (Verification of the only other
assumption in the theorem, namely Assumption~\ref{asm:r-asm}, was
already illustrated in prior examples---see for instance 
Example~\ref{ex:Butterworth}.)

Throughout this section, we set
\[
  \Hc = L^2(\Omega), \qquad V = \mathring{H}^1(\Omega), \qquad \Ac = -\Delta,
\]
with $\dom (\Ac)= \{ u \in V: \Delta u \in L_2(\om)\} $.
Here $\om$ is a Lipschitz polygonal subset of $\RR^2$ and we have used
standard notation for Sobolev spaces, e.g., the space $L^2(\om)$
denotes the Hilbert space of square-integrable functions on \(\Omega\)
endowed with the standard inner product, and
\(\mathring{H}^1(\Omega)\) denote its Sobolev subspace functions with
square-integrable weak derivatives and vanishing trace on the boundary
\(\partial \Omega\).  Our interest is in approximating eigenspaces of
some targeted eigenvalue clusters of~$\Ac$.

Given any $f \in \Hc$, the application of the resolvent at a
$z \in \rho(\Ac)$ produces $u = R(z) f$ that solves the Dirichlet problem 
\begin{align}\label{eq:elliptic-resolvent}
  z u + \Delta u &= f \quad \text{in } \Omega, 
  & u & = 0 \quad \text{on } \partial \Omega.
\end{align}
Its weak formulation is to find  $u \in \Ho^1(\om)$ satisfying
\begin{equation}
  b_z(u, v) = (f, v)_{L^2}, \qquad \tForAll v \in \mathring{H}^1(\Omega),
  \label{eq:elliptic-weak}
\end{equation}
where the sesquilinear form \(b_z\) is defined by
\begin{equation}
  b_z(u, v) = z (u, v)_{L^2}
  -(\nabla u, \nabla v)_{L^2}
  \label{eq:elliptic-bilinear}
\end{equation}
for any $ u, v \in \mathring{H}^1(\Omega).$ Here and throughout,
$(\cdot, \cdot)_{L^2}$ and $\| \cdot \|_{L^2}$ denote the inner
product and norm, respectively, on $L^2(\om)$ or its Cartesian
products.  The inf-sup condition holds for $b_z(\cdot, \cdot)$, i.e., 
\begin{equation}
  \label{eq:Laplace-inf-sup'}
  \sup_{v \in \mathring{H}^1(\Omega)}
  \frac{b_z(w, v)}{\| \nabla v \|_{L^2}}
  \geq
  \frac{1}{\beta_z}
  \norm{\nabla w}_{L^2
  }, 
\end{equation}
holds for all $w \in \mathring{H}^1(\Omega) $ with
$ \beta_z = \sup_{\lambda \in \sigma(\Ac)} \left\lvert
  \frac{\lambda}{\lambda - z} \right\rvert$
as shown, e.g.,  in~\cite[Lemma 3.1]{GopGruOvaPar2019-1}.
 Moreover, by our
assumption on the boundary of $\om$, it consists of smooth segments
intersecting at finitely many vertices and there is some
$1/2 < \alpha$ such the largest interior angle subtended at such
vertices is $\pi/\alpha$.  Then for any $s < \alpha$, standard
elliptic regularity theory \cite{Grisv85} shows that there is a
constant $\Cregz >0$ such that
\begin{equation}
  \label{eq:regularity}
  \| R(z) f \|_{H^{1+s}(\om)} \le \Cregz \| f \|_{\Hc},
  \qquad \text{ for all }  f \in \Hc \text { and } z \in \rho(\Ac).
\end{equation}
This gives the stability of solutions of~\eqref{eq:elliptic-weak}.
Note that \(V=\Ho^1(\om)\) is endowed with the $H^1(\om)$-norm, so it is continuously embedded into \(\Hc\), since \(\|u\|_{\Hc} = \|u\|_{L_2} \leq \|u\|_{H^1(\Omega)}\)
for all \(u \in V\).
Moreover, \(R(z)V \subset V\) for all \(z \in \rho(\Ac)\), since   \(\dom (\Ac)\) is contained in \(V\) in this example.

We now proceed to discuss discretization by FOSLS (First-Order Systems
Least Squares) and DPG (Discontinuous Petrov-Galerkin) methods. The
FOSLS method was developed in~\cite{CaiEtAl1994-1, CaiEtAl1997-1} and
a detailed account of its development can be found in the
book~\cite{BocheGunzb09a}. The DPG method and its error estimators
were developed in~\cite{DemkoGopal11,
  DemkoGopalNiemi12,CarDemGop2014-1} and a recent
review~\cite{DemGop2025-1} summarizes these developments.  In both
cases, we use \(\Omega_h\), a conformal simplicial finite element mesh
of \(\Omega\), that is decomposed into different elements \(K\).
Given \(K \subset \Omega\) and an integer \(k \geq 0\), we denote by
\(P^k(K)\) the space of polynomials of degree at most \(k\) defined on
\(K\), and by \(P_\mathsf{H}^k(K)\) the space of homogeneous
polynomials of degree \(k\) defined on \(K\).  The local
Raviart-Thomas space on an element $K$ equals
\(RT^k(K) = P^k(K)^d \oplus P_\mathsf{H}^k(K) {x}\), where
\({x} = (x_1, \ldots, x_d)\) is the position vector in \(K\).  We
adopt the convention of writing $B_1 \lesssim B_2$ whenever the 
inequality $B_1 \le C B_2$ holds with some mesh-independent constant $C$
whose value at different occurrences may differ.

\subsection{Eigenspace error estimation using FOSLS estimators}
\label{ssec:FOSLS}

The FOSLS discretization of the resolvent operator of $\Ac$ starts
with a first-order reformulation of~\eqref{eq:elliptic-resolvent} by
introducing the flux $q = - \nabla u$.
Let $X:=  H(\diver) \times \mathring{H}^1(\Omega)$.
Then  defining  the first-order operator 
$  \Af_z: X \to 
  L_2(\Omega)^2 \times L_2(\Omega)
$ by 
\[
  \Af_z (q, u) = ( q + \nabla u,
  \;
  -\diver q + z u)
\]
the resolvent problem~\eqref{eq:elliptic-resolvent} takes the form
\begin{equation}
  \label{eq:14}
  (q, u) \in X: \qquad 
  \Af_z(q, u) = (0, f),
\end{equation}
for any $f \in \Hc$. This identifies, in addition to $R(z) f$, a mapping into the
flux component, which we denote by $R^q(z) f$, i.e., 
\[
  \begin{tikzcd}
    f 
    \ar[r, mapsto, "R(z)"]
    &
    {u\text{ solving~\eqref{eq:14}}},
    &
    \qquad
    &
    f
    \ar[r, mapsto, "R^q(z)"]
    & {q\text{ solving~\eqref{eq:14}.}}
  \end{tikzcd}
\]
For any integer 
\(k \geq 1\), consider the Lagrange and the
Raviart-Thomas finite element spaces given by
\begin{align*}
  V_h & =  \set*{
    v \in \mathring{H}^1(\Omega) : v\restriction_K \in P^k (K),
    \tForAll K \in \Omega_h
  }, \\
  RT_h & = \set*{
    q \in H(\diver,\Omega) : q\restriction_K \in RT^{k-1} (K),
    \tForAll K \in \Omega_h
  },
\end{align*}
respectively (see, e.g.,~\cite{ErnGue2021-2, BreFor2012-1}) and put
\(X_h = RT_h \times V_h\). 

The two ingredients needed in our framework, namely $R_h(z)$ and
$\EE$, are now set as follows.
Set the  approximate resolvent operator
$R_h(z): \Hc \to V_h$ as the second component of the unique
$(q_h, u_h) \in X_h$ that solves the FOSLS equation
\begin{equation}
  \label{eq:FOSLS-normal-eq}
  \left(
    \Af_z (q_h, u_h), \Af_z (r_h, v_h)
  \right)_{L_2}
  = \left(
    (0, f), \Af_z (r_h, v_h)
  \right)_{L_2}, 
\end{equation}
for all $ (r_h, v_h) \in X_h,$ i.e., $R_h (z) f := u_h$. Let us also
denote the flux component of the solution by $R_h^q(z) f := q_h$, i.e., 
\[
  \begin{tikzcd}
    f 
    \ar[r, mapsto, "R_h(z)"]
    &
    {u_h\text{ solving~\eqref{eq:FOSLS-normal-eq}}},
    &
    \qquad
    &
    f
    \ar[r, mapsto, "R^q_h(z)"]
    & {q_h\text{ solving~\eqref{eq:FOSLS-normal-eq}.}}
  \end{tikzcd}
\]
Next, to define the error estimator \(\EE:V_h \to Y\), setting
\(Y = \left[ L_2(\Omega)^2 \times L_2(\Omega) \right]^\nquad\), we
consider, for each pole $z_k$ of the rational function
in~\eqref{eq:feast-filter}, an operator
$\EE_k: V_h \to L_2(\Omega)^2 \times L_2(\Omega)$ defined by
\begin{subequations}
  \label{eq:FOSLS-EE}
  \begin{align}
    \EE_k v_h & = (0, v_h) - \Af_{z_k} (R_h^q(z_k) v_h, R_h(z_k) v_h)              
  \end{align}
  for any \(v_h \in V_h\) and set 
  \begin{equation}
    \label{eq:FOSLS-EE-defn}
    \EE v_h  = \left( \EE_1 v_h, \ldots, \EE_\nquad v_h \right) \in Y.
  \end{equation}
\end{subequations}
Note that
\( \norm{\EE v_h}_{Y}^2 = \sum_{k=1}^\nquad \norm{\EE_k
  v_h}^2_{L_2} 
\).  With these definitions in place, we now proceed to verify 
Assumptions~\ref{asm:resolvent-approx} and~\ref{asm:source-ee}.

\begin{lemma}
  \label{lem:Laplace-elliptic'}
  For any \(z \in \rho(\Ac)\), there exists \(\gamma_z > 0\) such that, for all
  \(u \in \mathring{H}^1(\Omega)\) and \(q \in H(\diver)\), we have
  \begin{equation}
    \label{eq:Laplace-elliptic'}
    \norm{q}_{L_2
    }^2 +
    \norm{u}_{H^1(\Omega)
    }^2 
    \leq \gamma_z \norm{\Af_z (q, u)}_{L_2}^2.
  \end{equation}
  Moreover, the operator \(\Af_z\) is a continuous bijection.
\end{lemma}
\begin{proof}
  Let \(u \in \mathring{H}^1(\Omega)\) and \(q \in H(\diver)\) and let
  \(g \in L_2(\Omega)^2\) and \(f \in L_2(\Omega)\) be such that
  \(\Af_z (q, u) = (g, f)\). This is a  system of two equations, from which 
  eliminating $q$, we obtain 
  $
    z u - \diver(g - \nabla u) = f.
  $
  Rearranging,
  \[
    (z + \Delta) u  = f + \diver g,
  \]
  which, due to the boundary condition on $u$, has the following
  associated weak formulation
  \[
    b_z(u,v) = (f + \diver g)(v), \qquad \tForAll v \in \mathring{H}^1(\Omega),
  \]
  where, since the distributional divergence of $g$ is in $H^{-1}(\om)$, we have treated $f + \diver g$ as a bounded linear functional on $\Ho^1(\om)$.
  By the inf-sup condition~\eqref{eq:Laplace-inf-sup'} and
  the Poincaré inequality, we have
  \begin{align*}
    \norm{ u}_{H^1(\om)}
     & \lesssim \sup_{v \in H^1(\Omega)}
     \frac{b_z(u, v)}{\norm{v}_{\mathring{H}^1(\Omega)}} 
       = \sup_{v \in H^1(\Omega)}
       \frac{(f + \diver g)(v)}{\| v \|_{H^1(\Omega)}}
    \\
     & \leq \norm{f}_{H^{-1}(\Omega)} + \norm{\diver g}_{H^{-1}(\Omega)}
       \lesssim \norm{f}_{L_2}+ \norm{g}_{L_2}
    \\
     & \lesssim \norm{\Af_z (q, u)}_{L_2}. 
  \end{align*}
  Since $\Ac(q, u) = (g, f)$ also  implies that $q + \nabla u = g$, the above inequality also implies
  \begin{align*}
    \norm{q}_{L_2}
&       \leq \norm{g}_{L_2} + \norm{\nabla u}_{L_2}
      \lesssim \norm{f}_{L_2}+ \norm{g}_{L_2}
    \\
    & 
      \lesssim \norm{\Af_z (q, u)}_{L_2}^2. 
  \end{align*}
  This proves~\eqref{eq:Laplace-elliptic'}.  In particular, it shows
  that the operator \(\Af_z\) is continuous and injective.  The
  equivalence of the weak formulation~\eqref{eq:elliptic-weak} and the
  boundary value problem~\eqref{eq:elliptic-resolvent} implies that
  \(\Af_z\) is also surjective.
\end{proof}

\begin{proposition}
  The FOSLS resolvent approximation $R_h(z)$ converges in operator norm on $V$,
  i.e., Assumption~\ref{asm:resolvent-approx} holds.
\end{proposition}
\begin{proof}
  Let $f \in V$,  $u = R(z) f$,  and $u_h = R_h(z) f$.
  Since $V_h \subset V$, by Galerkin orthogonality
  \begin{align*}
    \norm*{\Af_z (q - q_h, u - u_h)}_{L_2}^2
    & =
      \left(
      \Af_z(q - q_h, u - u_h), \Af_z(q - r_h, u - w_h)
      \right)_{L_2}
    \\
    & \lesssim
      \| \Af_z(q - q_h, u - u_h)\|_{L_2}
      \Big(
      \| q - r_h \|_{H(\diver)}  + \| u - w_h \|_{H^1(\om)} 
      \Big).
  \end{align*}
  By standard finite element best approximation estimates for the Lagrange
  and Raviart-Thomas spaces, this leads to 
  \begin{equation}    
    \label{eq:7-Az}
    \norm*{\Af_z (q - q_h, u - u_h)}_{L_2}
    \lesssim h^r |u |_{H^{r+1}(\om} + h^r |q|_{H^r(\om)} + h^r | \diver q |_{H^r(\om)}
  \end{equation}
  for some $0 < r \le 1$. Choosing $r\le s$, where $s$ is as
  in~\eqref{eq:regularity}, the seminorms on the right hand side above can
  be bounded as follows:
  \begin{gather*}
    |u|_{H^{r+1}(\om)} + |q|_{H^{r}(\om)}
    \lesssim \| u \|_{H^{s+1}(\om)} \lesssim
      \| f \|_{\Hc} \lesssim \| f \|_V,
    \\
    |\diver q|_{H^{r}(\om)}
    = |z u  - f|_{H^r(\om)} \lesssim \| f \|_{H^r(\om)} \lesssim \| f \|_V,
  \end{gather*}
  where in the last step we have used that $r \le 1$ as well as the
  first bound for~$u$.  Using these estimates in~\eqref{eq:7-Az},
  \[
    \norm*{\Af_z (q - q_h, u - u_h)}_{L_2}
    \lesssim\, h^r \| f \|_V.
  \]
  By Lemma~\ref{lem:Laplace-elliptic'},
  $    \norm*{u - u_h}_{H^1(\Omega)}^2
  + \norm*{q - q_h}_{L_2}^2
  \leq
  \gamma_z \norm*{\Af_z (q - q_h, u - u_h)}_{L_2}^2,
  $
  so we conclude that for any $f \in V$, 
  \begin{align}
    \label{eq:8}
    \| R(z) f - R_h(z) f\|_V
    & = \| u - u_h \|_{H^1(\om)}
      \le \gamma_z^{1/2} \| \Af_z( q - q_h, u - u_h) \|_{L_2} 
      \lesssim h^r \| f \|_V,    
  \end{align}
  thus verifying Assumption~\ref{asm:resolvent-approx} as $h \to 0$. 
\end{proof}

\begin{proposition}
  The FOSLS error estimator $\EE$ in \eqref{eq:FOSLS-EE} satisfies
  Assumption~\ref{asm:source-ee}.
\end{proposition}
\begin{proof}
  Let $v_h \in V_h$ and $z \in \rho (\Ac)$. Then
  $(q, u) = (R^q(z) v_h, R(z) v_h) \in X$ satisfies
  \begin{equation}
    \label{eq:17}
      \Af_z(q, u) = (0, v_h).
  \end{equation}
  Let
  $(q_h, u_h) = (R_h^q(z) v_h, R_h(z) v_h) \in X_h$.  
  By Lemma~\ref{lem:Laplace-elliptic'}, as in~\eqref{eq:8}, 
  \begin{align*}
    \| R(z) f - R_h(z) f\|_V^2
    & = \| u - u_h \|_{H^1}^2
    \\
    &
      \le \gamma_z
      \left\| \Af_z( q - q_h, u- u_h) \right\|_{L_2}^2
    && \text{ by Lemma~\ref{lem:Laplace-elliptic'}}
    \\
    &
      = \gamma_z
      \left\| (0, v_h) -\Af_z( q_h, u_h) \right\|_{L_2}^2
    && \text{ by~\eqref{eq:17}}
    \\
    & = \gamma_z
    \left\| (0, v_h) -\Af_z\big( R_h^q(z) v_h, R_h(z) v_h\big) \right\|_{L_2}^2.      
  \end{align*}
  When this estimate is applied with $z = z_k$, the last term equals
  $\| \EE_k v_h \|_{L_2}^2$. Hence collecting the estimates for all
  the poles $z_k$, by triangle inequality, we obtain
  $ \| S_r v_h - \Srh v_h \|_V \lesssim \| \EE v_h\|_{Y}, $ which
  verifies Assumption~\ref{asm:source-ee}.
\end{proof}

\subsection{Eigenspace error estimation using DPG estimators}
\label{ssec:DPG}

To define the DPG discretization of the resolvent applied to an
$f \in \Hc$, reconsider the form $b_z(u,v)$
in~\eqref{eq:elliptic-bilinear} but now  computing derivatives of $v$ element-by-element and thereby extending it  to  $v$ in the space of piecewise polynomials of degree $k+3$ (without any interlement continuity),  namely for $v$ in
\[
  Y_h = \{ v \in L_2(\om): v\restriction_K
  \in P^{k+3}(K) \tForAll K \in \Omega_h\}.
\]
This is a subspace of the ``broken'' $H^1$ space $H^1(\om_h):= \prod_{K \in \om_h} H^1(K)$ where the inner product is 
\[
  (v, y)_{H^1(\om_h)} = \sum_{K \in \om_h} \int v\overline y\; dx
  +
  \int_K \nabla v \cdot  \nabla \overline y \; dx,
\]
for any $v, y \in Y_h$, where again, the derivatives are computed
element by element.  Let $n$ denote the unit outward normal vector on
element boundaries and let $\hat X_h = \{ \hat r_n: $ on the boundary
of every $K \in \om_h$,
$\hat r_n|_{\partial K} = (r \cdot n)|_{\partial K}$ for some
$r \in RT_h\}$.  Using the prior form $b_z(\cdot, \cdot)$ from \eqref{eq:elliptic-weak}, but now
extended to $V \times Y_h$ as mentioned above, we define
\[
  a_z((w, \hat r_n), v)
  = b_z(w, v) + \ip{\hat r_n, v}_h
\]
where
\[
  \ip{ \hat r_n, v}_h = \sum_{K \in \om_h} \int_{\partial K} \hat r_n \, \overline v \, ds.
\]
Given any \(f \in L_2(\Omega)\), the DPG discretization of the
resolvent problem $(z - \Ac) u = f$ finds a
$u_h \in V_h, \hat q_n \in \hat X_h$, and \(\varepsilon_h \in Y_h\)
such that
\begin{subequations}
\label{eq:elliptic-DPG}
\begin{align}
  \parens{\varepsilon_h, \delta_h}_{H^1(\Omega_h)} +
  a_z((u_h, \hat q_n), \delta_h)
  & = (f, \delta_h)_{L_2},
  \\
  \overline{a_z((w_h, \hat r_n), \varepsilon_h)}
  & = 0,
\end{align}
\end{subequations}
for all $\delta_h \in Y_h$ and
$ (w_h, \hat r_n) \in V_h \times \hat X_h.$ This system, known as the
primal DPG method, is well known to be uniquely solvable.  The mapping from $f$ 
to the solution component in $V_h$
defines the DPG resolvent approximation $R_h (z): \Hc \to V_h$. In addition, we also need the  map to the $\varepsilon_h$ solution component, which we denote
by $R_h^\varepsilon (z) f := \varepsilon_h$, i.e.,
\[
  \begin{tikzcd}
    f 
    \ar[r, mapsto, "R_h(z)"]
    &
    {u_h\text{ solving~\eqref{eq:elliptic-DPG}}},
    &
    \qquad
    &
    f
    \ar[r, mapsto, "R^\veps_h(z)"]
    & {\veps_h \text{ solving~\eqref{eq:elliptic-DPG}}}.
  \end{tikzcd}
\]
Using the resolvent approximation at each $z_k$, we define
the DPG error estimator by
\begin{equation}
  \label{eq:DPG-EE}
  \EE v_h  = \left( R_h^\varepsilon (z_1)  v_h, \ldots,
    R_h^\varepsilon (z_\nquad)  v_h \right).  
\end{equation}
With $Y$ set to the $N$-fold Cartesian product $Y = Y_h^N$, normed by
the product $H^1(\oh)$-norm defined by 
\begin{equation}
\label{eq:element-by-element-Y-dpg}
  \| (y_1, \dots, y_N)  \|_Y^2 = \sum_{k=1}^N \| y_k\|_{H^1(\oh)}^2
  = \sum_{k=1}^N\sum_{K \in \oh}  \| y_k\|_{H^1(K)}^2, 
\end{equation}
equation~\eqref{eq:DPG-EE}  defines $\EE: V_h \to Y$ and completes
the description of all DPG ingredients needed to fit the prior
framework.

\begin{proposition}
  The DPG resolvent approximation converges in operator norm on $V$
  and Assumption~\ref{asm:resolvent-approx} holds. Furthermore, the
  DPG error estimator $\EE$, set in \eqref{eq:DPG-EE}, satisfies
  Assumption~\ref{asm:source-ee}.
\end{proposition}
\begin{proof}
  The proof of the first statement is very similar to the proof of
  \cite[Lemma~3.4]{GopGruOvaPar2019-1} so we omit it. To prove the
  statement on $\EE$, we appeal to~\cite[Theorem~6.4]{DemGop2025-1}
  (see also~\cite{CarDemGop2014-1}), noting that all assumptions
  required for that result are also essentially verified there (see
  e.g., \cite[Example~4.2]{DemGop2025-1}). Hence the conclusion of
  that theorem on reliability of DPG error estimators give
  \begin{equation}
    \label{eq:19}
    \| R(z) v_h - R_h(z) v_h \|_{V} 
    \lesssim  \| \vpi \| \,\| R_h^\veps(z) v_h \|_{H^1(\om_h)} + \osc(v_h)
  \end{equation}
  for any $z \in \rho(\Ac)$ and $v_h \in V_h$, where 
  \[
    \osc(v_h) = \sup_{y \in H^1(\om_h)}
    \frac{ (v_h, y - \vpi y )_{L_2}}{ \| y \|_{H^1(\om_h)}}.
  \]
  Here $\vpi: H^1(\oh) \to Y_h$ is a Fortin operator satisfying
  $a_z( (w_h, \hat r_n), y - \vpi y )=0$ for all $y \in H^1(\oh)$, $w_h \in V_h$ and
  $ \hat r_n \in \hat X_h$, whose operator norm admits a
  meshsize-independent bound, i.e.,  $\|\vpi \| \lesssim 1$, and satisfies the
  moment condition
  \begin{align}
    \label{eq:20}
    \int_K (\vpi y - y) q \, dx & = 0, \qquad \text{ for all } q \in P^k(K),
  \end{align}
  on all elements $K \in \om_h$,
  as well as further moment conditions on  mesh edges---see e.g.,
  \cite[Theorem~5.4]{DemGop2025-1} or \cite{GopalQiu14} for further details.
  Observe that
  for any $v_h \in V_h$,
  \[
    (v_h, y - \vpi y )_{L_2} = 0 
  \]
  due to~\eqref{eq:20}. Hence $\osc(v_h) = 0$ and~\eqref{eq:19} implies
  \[
    \| R(z_k) v_h - R_h(z_k) v_h \|_{V} 
    \lesssim\| R_h^\veps(z_k) v_h \|_{H^1(\oh)}.
  \]
  Hence the definition of $\EE$ in~\eqref{eq:DPG-EE} and triangle
  inequality completes the verification of
  Assumption~\ref{asm:source-ee}.
\end{proof}

\section{Numerical results}

\label{sec:numerical-results}

In this section, we briefly present two numerical examples. The first
concerns  a selfadjoint Laplace eigenproblem on a interesting domain and falls
within what is covered by the theory in Section~\ref{sec:application}.
Here we report results on how the adaptive algorithm applied to two distinct
eigenvalue clusters with different regularity profiles captures distinct refinement patterns for each.
The second is a computational study of the proposed eigenspace error estimation
technique applied to a nonselfadjoint  Helmholtz eigenproblem for
leaky waveguide modes. Here, the adaptive algorithm finds an almost
circularly symmetric refinement pattern even though none  of
the two eigenmodes found have circular symmetry, indicating again
that refinement patterns target the cluster as a whole and not solely individual eigenfunctions.

In both examples, we consider  a discretization
using the DPG method described in
Section~\ref{ssec:DPG}. We use the element-by-element definition of
previously described eigenspace error estimator  to design an $h$-adaptive mesh refinement algorithm in a standard fashion  based on the
SOLVE$\rightarrow$ESTIMATE$\rightarrow$MARK$\rightarrow$REFINE paradigm (see, e.g.,~\cite{Ver1994-1}), described further below.
For adaptive mesh refinement, finite element assembly, and visualization, we use 
the NGSolve finite element library~\cite{Sch2025-1} and its implementation of
Lagrange finite elements of polynomial degree  \(k=5\).
For an eigensolver, 
we use a simple Python implementation of
the FEAST algorithm (available in~\cite{GopEtAl2025-1}), which 
generates a sequence of eigenspace iterates that converge
to the eigenspace approximation~$E_h$, represented computationally by
a set of basis vectors returned upon meeting a stopping criterion.


For the marking step in adaptivity,
we employ a greedy strategy, based on the error estimator
defined in Section~\ref{sec:eigenspace-error-estimate}.
It is necessary to  ``localize'' the error estimator \(\EE: V_h \to Y\) to each element \(K\) for this purpose. In all our examples, there is a natural
space $Y(K)$ on each $K$ such that 
\[
  Y = \prod_{K \in \oh} Y(K),
\]
allowing $\EE$ to be split into local contributions.  For instance,
\eqref{eq:element-by-element-Y-dpg} in the DPG case shows that the
norm of any $y = (y_1, \dots, y_N) \in Y$ with $y_k \in H^1(\oh)$ can
be written using
\[
  \| y \|_{Y(K)}^2 := \sum_{k=1}^N \| y_k \|_{H^1(K)}^2
\]
as $\| y\|_Y^2 = \sum_{K \in \oh} \| y\|_{Y(K)}^2$. Similarly, in the FOSLS 
case of~\eqref{eq:FOSLS-EE}, $y = \EE v_h$ has $N$ components $y_k = \EE_k v_h$, $k=1,\dots, N$, and we set $Y(K) = L_2(K)^N$ normed by 
\[
  \| y \|_{Y(K)}^2 := \sum_{k=1}^N \| y_k \|_{L_2(K)}^2.
\]
Using the local $Y(K)$ in each case, we define element-by-element
error indicators $\eta_K$ and accompanying quantities, namely let 
\begin{equation*}
  \eta_K = \norm*{\EE v_h}_{Y(K)}, \quad
  \eta_{\max} = \max_{K \in \Omega_h} \eta_K, \quad \tAnd
  \eta_{\ell^2} = \left( \sum_{K \in \Omega_h} \eta_K^2 \right)^{1/2},
\end{equation*}
be the local error indicator, the maximum error indicator, and the
\(\ell^2\)-norm of the error indicators, respectively.
Next, we say \(K\) is marked for refinement if, given a fixed ratio
parameter \(0 < \theta < 1\), \(\eta_K \geq \theta \ \eta_{\max}\).
Finally, we refine all marked elements using a standard conforming long-edge
bisection algorithm.
We set \(\theta = 0.9\).
Various other  marking strategies can be
employed, such as the Dörfler marking strategy~\cite{Dor1995-1}.
However, we note that the greedy strategy is simple to implement and
performs well in practice.


\subsection{Laplace eigenproblem on a Gordon-Webb-Wolpert drum}

Our first numerical example is the aforementioned Laplace eigenproblem on
one of the Gordon-Webb-Wolpert (GWW) isospectral
drums~\cite{GorWebWol1992-1, Dri1997-1}.  This domain is one of two
non-congruent planar domains that share the same Dirichlet Laplace
spectrum (providing a negative answer to the famous question ``Can one
hear the shape of a drum?'' posed by Kac~\cite{Kac1966-1}).
We begin by computing a cluster of eigenvalues surrounding the
ninth eigenvalue.
We define a circular contour \(\Gamma_\mathsf{cl}\) of
center \(c_\mathsf{cl} = 12.33\) and radius \(r_\mathsf{cl} = 1.00\)
which contains the eighth, ninth and tenth eigenvalues of the Laplace operator
associated with both GWW domains.
As reported in~\cite{Dri1997-1},
the eighth, ninth and tenth eigenvalues of the Laplace operator on both GWW
domains are approximately given by
\begin{align*}
  \lambda_8 & \approx 11.5413953956, &
  \lambda_9 & \approx 12.3370055014, & \quad \text{and} \quad
  \lambda_{10} & \approx 13.0536540557.
\end{align*}
We denote the cluster of eigenvalues contained in \(\Gamma_\mathsf{cl}\)
by \(\Lambda_\mathsf{cl} = \{\lambda_8, \lambda_9, \lambda_{10}\}\).
Similarly, we define a circular contour \(\Gamma_\mathsf{sg}\) of
center \(c_\mathsf{sg} = 12.33\) and radius \(r_\mathsf{sg} = 0.4\)
which contains only the ninth eigenvalue, and denote the corresponding
cluster by \(\Lambda_\mathsf{sg} = \{\lambda_9\}\).
Then we use the FEAST method that discretizes the contour integral by
a four-point trapezoidal quadrature rule producing a rational function
as in Example~\ref{ex:Butterworth} with $N=4$ (including the shift
$\phi$ mentioned there to avoid poles on the real axis).
Using the corresponding DPG error estimator $\EE$, we run the adaptive
algorithm.
For the purpose of benchmarking, we also ran the adaptive algorithm using
an explicit residual error estimator~\cite[Section 3.2]{Lar2000-1}
and a DWR error estimator~\cite[Equation (85)]{HeuRan2001-1} (with a
first-order weight given by the gradient of the computed eigenfunction).
Note that these estimators were not designed for a cluster of eigenvalues,
but we have adapted them to this setting by considering the
element-wise sum of the error indicators for each eigenpair
in the cluster.


The result, displayed in Figure~\ref{fig:meshes}, shows
that the method is capable of automatically finding the right
locations needing refinement.
The initial mesh in Figure~\ref{fig:mesh1} is a quasi-uniform mesh
with appoximate mesh size \(h = 0.3\).
The final mesh in Figure~\ref{fig:mesh_ref_cluster}
shows strong refinementnear the re-entrant corners of
the domain, where the eigenfunctions are known to be singular or to have
larger gradients (see Figure~\ref{fig:ews}).
In contrast, the final mesh in Figure~\ref{fig:mesh_ref_single}
shows uniform refinement in the whole domain, due to the regularity 
of the ninth eigenfunction (see Figure~\ref{fig:ew9}).


\begin{figure}
    \centering
    \begin{subfigure}[t]{0.32\textwidth}
        \centering
        \includegraphics[width=\textwidth]{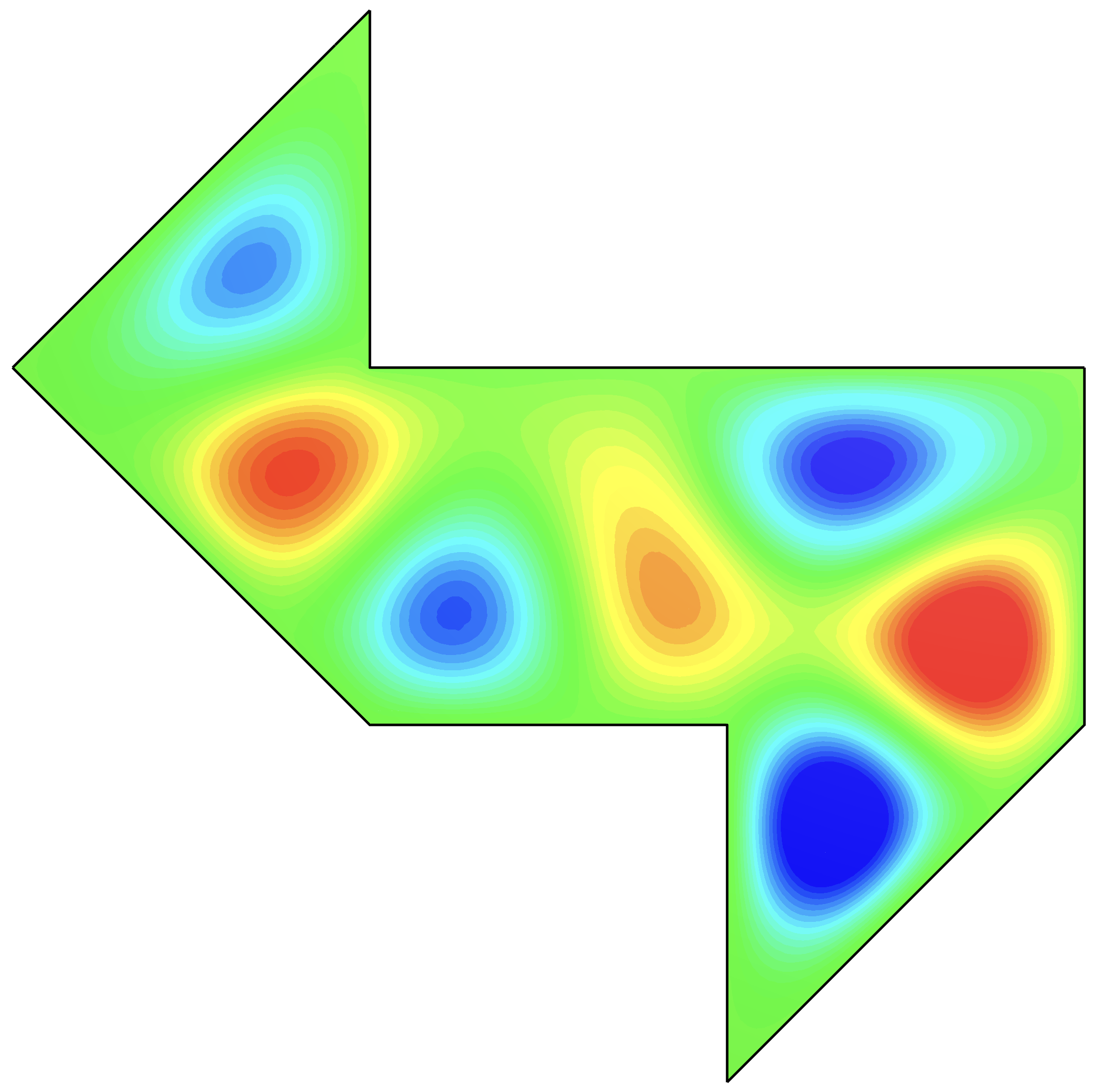}%
        \caption{Eighth eigenfunction.}
        \label{fig:ew8}
    \end{subfigure}%
    \hfill
    \begin{subfigure}[t]{0.32\textwidth}
        \centering
        \includegraphics[width=\textwidth]{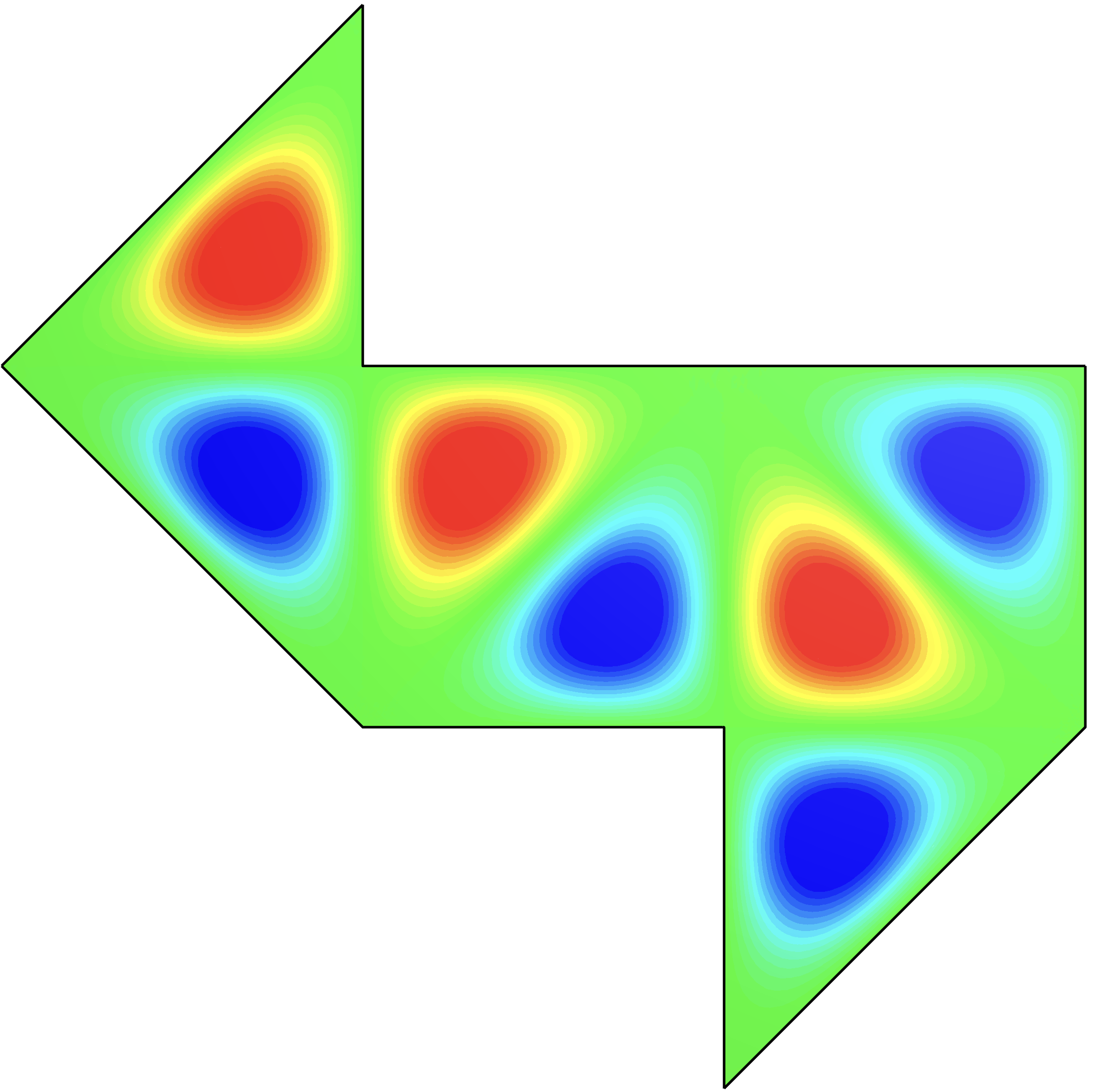}%
        \caption{Ninth eigenfunction.}
        \label{fig:ew9}
    \end{subfigure}%
    \hfill
    \begin{subfigure}[t]{0.32\textwidth}
        \centering
        \includegraphics[width=\textwidth]{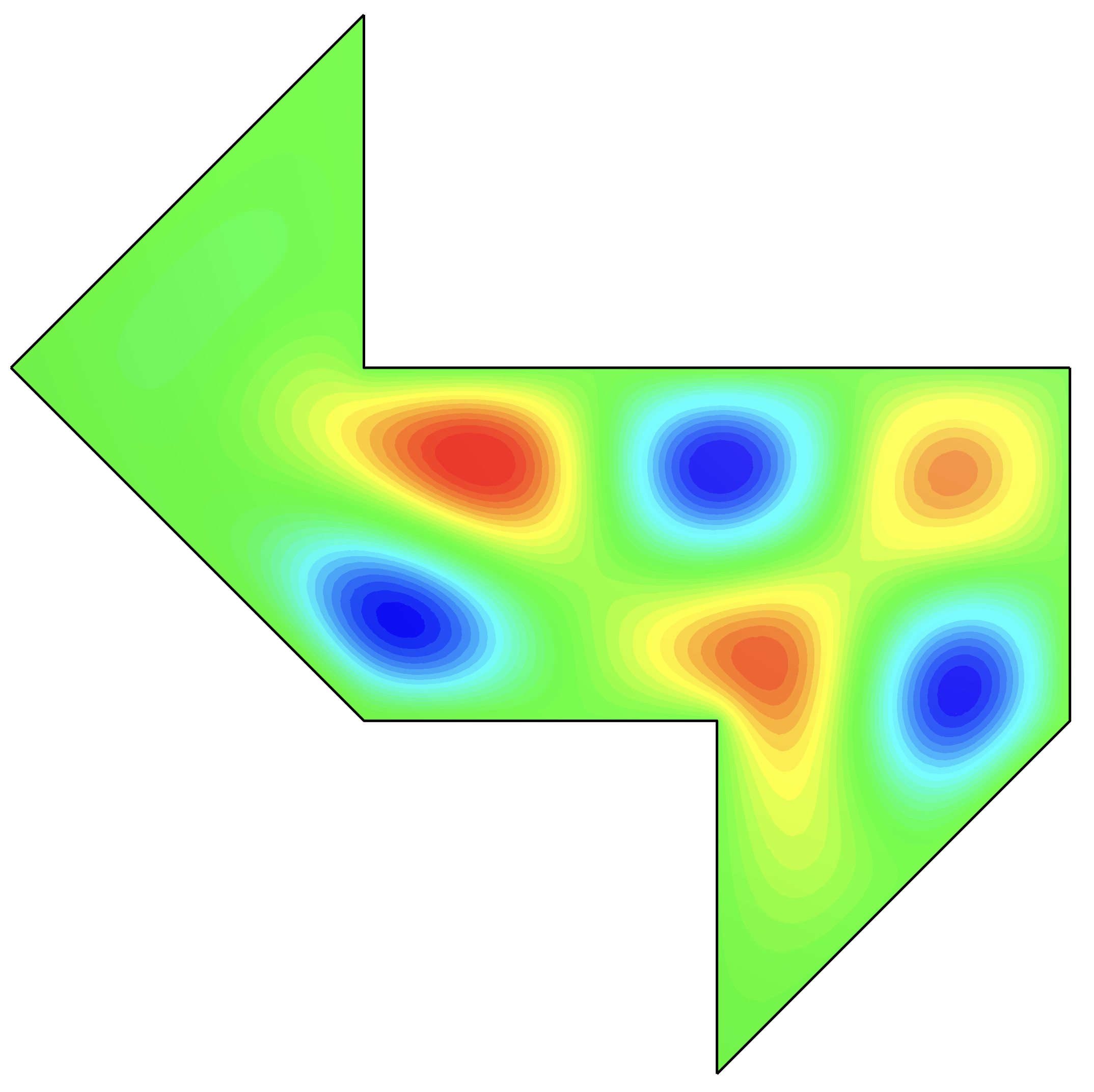}%
        \caption{Tenth eigenfunction.}
        \label{fig:ew10}
    \end{subfigure}%
    \caption{
        Eigenfunctions corresponding to the eighth, ninth and tenth eigenvalues
        ($\lambda_8$, $\lambda_9$ and $\lambda_{10}$) of the Laplace operator 
        a GWW isospectral drum.
        Scale varies between figures.
    }
    \label{fig:ews}
\end{figure}


\begin{figure}
    \centering
    \begin{subfigure}[t]{0.32\textwidth}
        \centering
        \includegraphics[width=\textwidth]{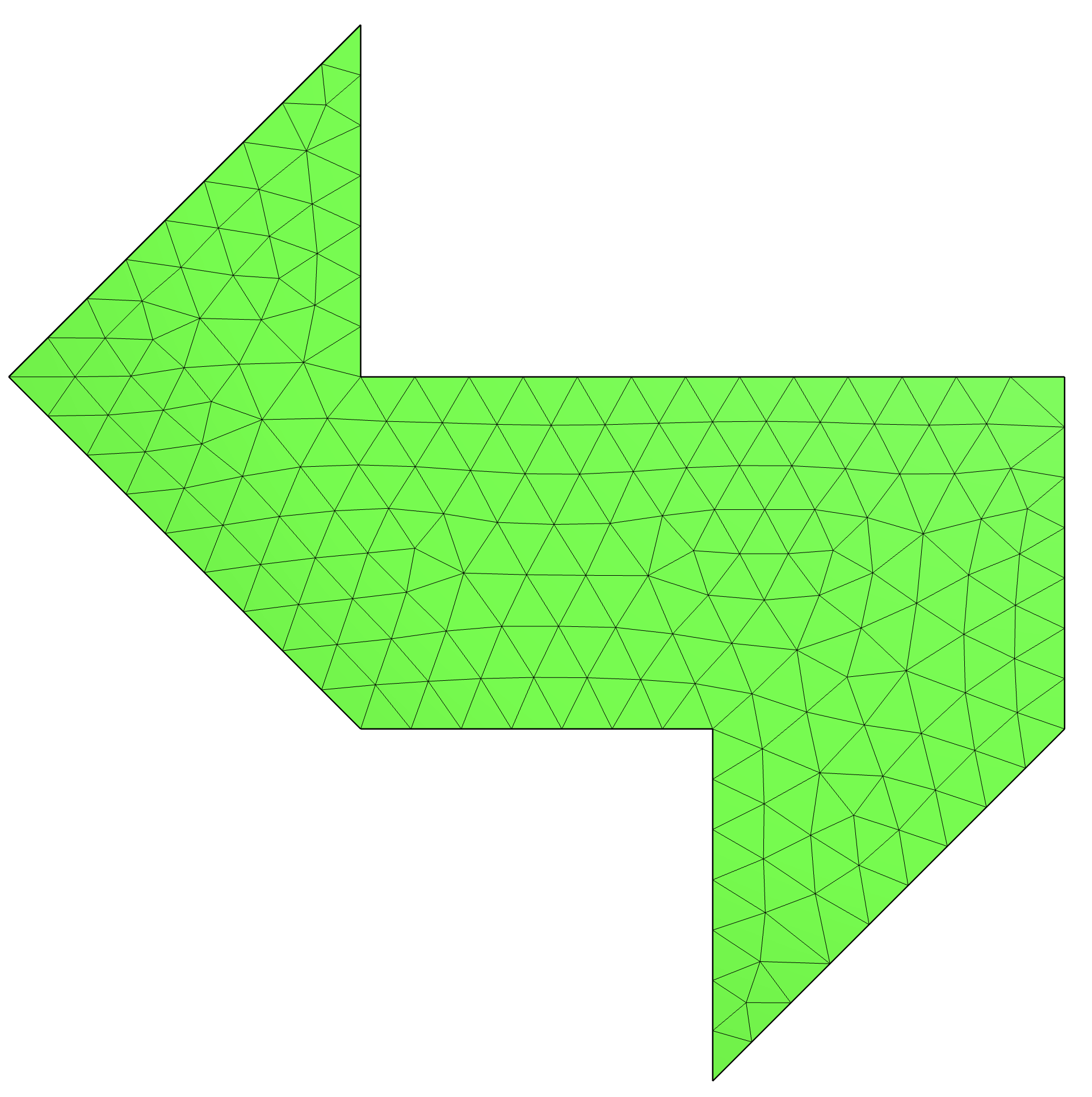}%
        \caption{Initial mesh.}
        \label{fig:mesh1}
    \end{subfigure}%
    \hfill
    \begin{subfigure}[t]{0.32\textwidth}
        \centering
        \includegraphics[width=\textwidth]{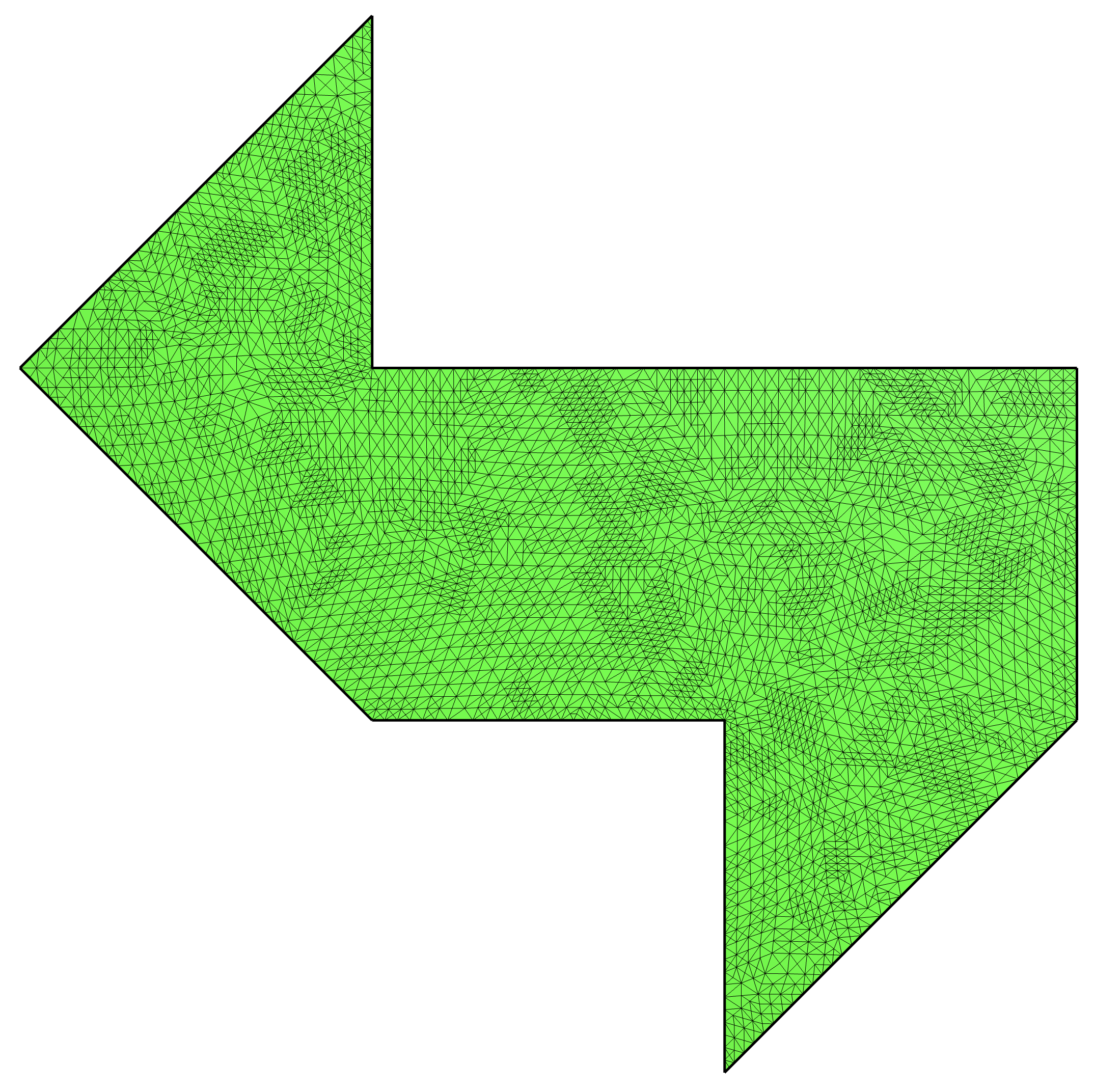}%
        \caption{Final mesh associated to \(\Gamma_\mathsf{sg}\).}
        \label{fig:mesh_ref_single}
    \end{subfigure}%
    \hfill
    \begin{subfigure}[t]{0.32\textwidth}
        \centering
        \includegraphics[width=\textwidth]{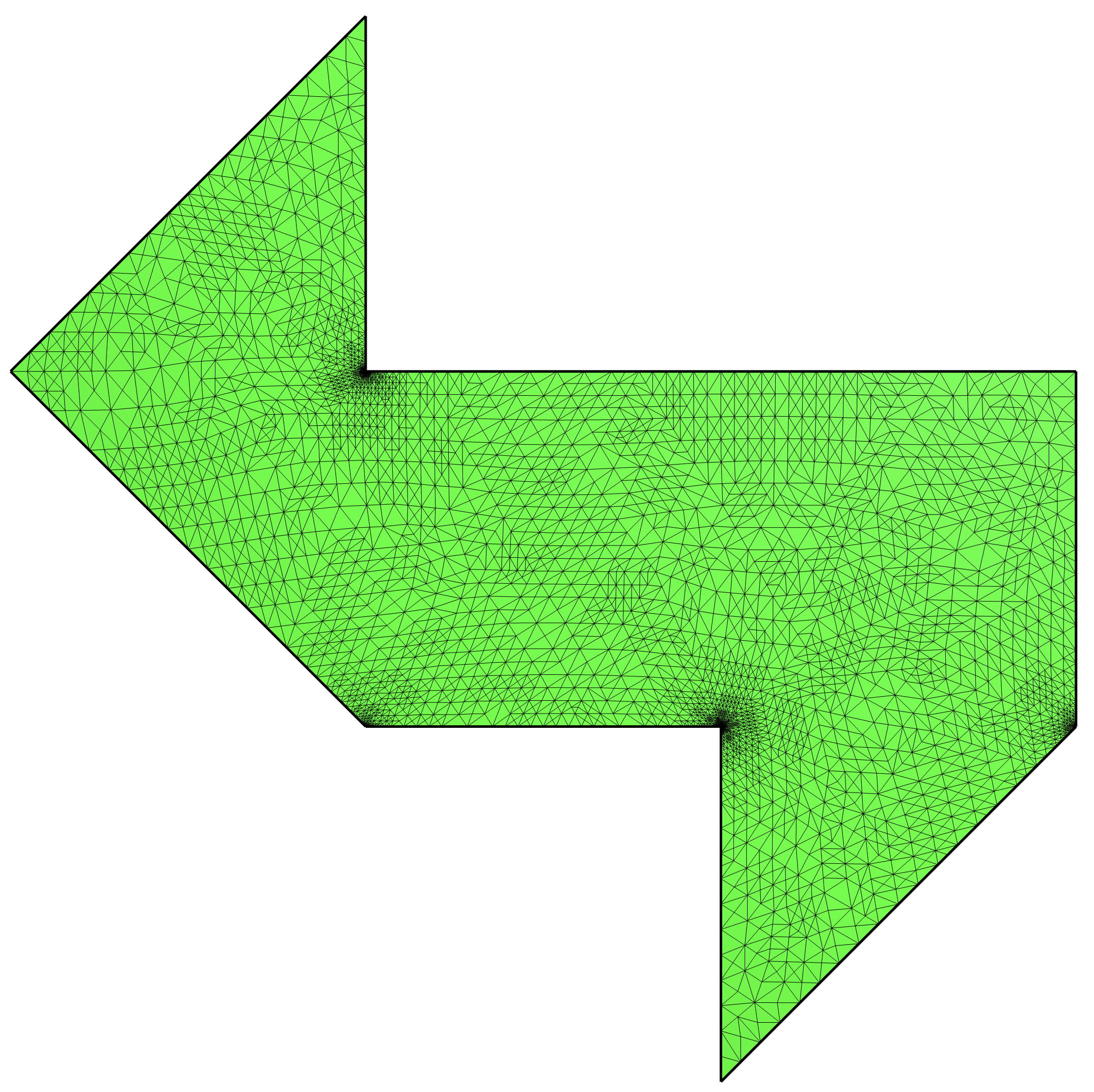}%
        \caption{Final mesh associated to \(\Gamma_\mathsf{cl}\).}
        \label{fig:mesh_ref_cluster}
    \end{subfigure}%
    \caption{
      On the left, initial mesh on a GWW isospectral drum.
      On the center and right, final meshes after the adaptive algorithm,
      using the DPG error estimator for computing 
      the single eigenvalue \(\lambda_9\) (center) and the cluster
      \(\Lambda_\mathsf{cl}\) (right).
    }
    \label{fig:meshes}
\end{figure}

\begin{figure}
  \centering
  \begin{tikzpicture}[scale=0.85]
    \centering
    \begin{loglogaxis}[
        cycle list name={three-by-two},
        xlabel={Number of \textsc{DOFs}},
        ylabel={Approximation error},
        legend columns=3,
        legend entries={%
            \(d_{\sf{cg-res}}\),
            \(d_{\sf{cg-dwr}}\),
            \(d_{\sf{dpg}}\),
            \(\eta_{\ell^2,\sf{cg-res}}\),
            \(\eta_{\ell^2,\sf{cg-dwr}}\),
            \(\eta_{\ell^2,\sf{dpg}}\),
        },
        table/x=ndofs,
    ]
      \addplot+[table/y=hausdorff, y filter/.expression={abs(y) < 0.5e-10 ? NaN : \pgfmathresult }] table {nine_single/cg_adapt_data.csv};
      \addplot+[table/y=hausdorff, y filter/.expression={abs(y) < 0.5e-10 ? NaN : \pgfmathresult }] table {nine_single/dwr_adapt_data.csv};
      \addplot+[table/y=hausdorff, y filter/.expression={abs(y) < 0.5e-10 ? NaN : \pgfmathresult }] table {nine_single/dpg_adapt_data.csv};
      \addplot+[table/y=l2_eta] table {nine_single/cg_adapt_data.csv};
      \addplot+[table/y=l2_eta] table {nine_single/dwr_adapt_data.csv};
      \addplot+[table/y=l2_eta] table {nine_single/dpg_adapt_data.csv};

      \draw[red, dashed, semithick] (axis cs:1e-2,1e-10) -- (axis cs:5e5,1e-10);

    \end{loglogaxis}
    \centering
  \end{tikzpicture}
  \caption{
    Hausdorff distance between the (approximated) ninth eigenvalue
    \(\Lambda_\mathsf{sg} = \{\lambda_9\}\) and its computed
    approximation, \(d_{\bullet}\),
    and \(\ell^2\)-norm of the error estimator
    \(\eta_{\ell^2,\bullet}\), for the different discretizations and error
    estimators
    (\(\bullet \in \{\mathsf{cg-res}, \mathsf{cg-dwr}, \mathsf{dpg}\}\)),
    against the number of \textsc{dofs}.
    The dashed red line indicates the level of accuracy of the reference eigenvalues.
  }
  \label{fig:poisson_gww1_single}
\end{figure}
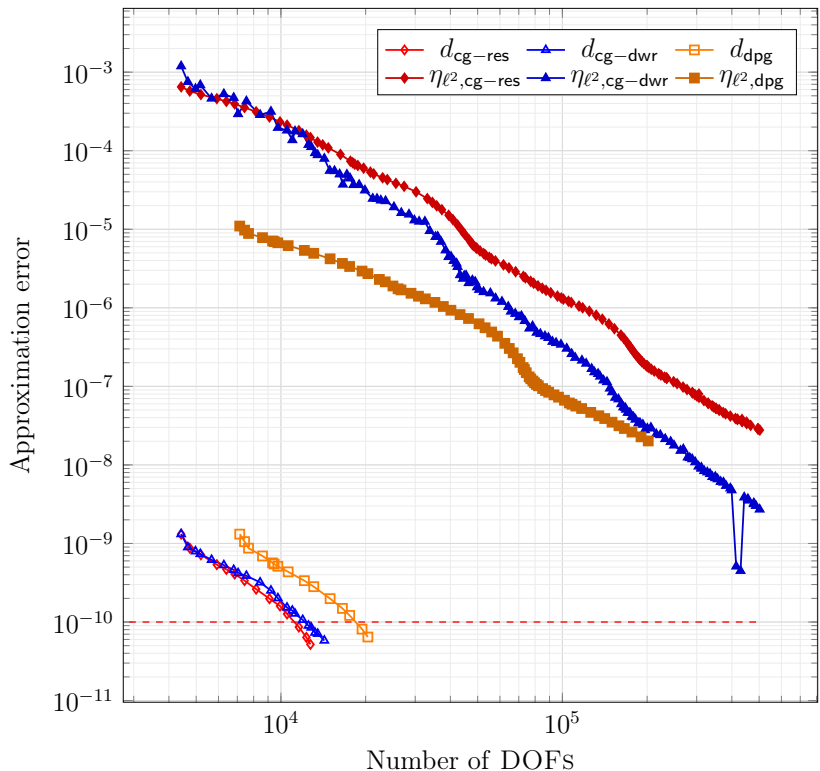

\begin{figure}
  \centering
  \begin{tikzpicture}[scale=0.85]
    \centering
    \begin{loglogaxis}[
        cycle list name={three-by-two},
        xlabel={Number of \textsc{DOFs}},
        ylabel={Efficiency},
        legend columns=3,
        legend entries={%
            \(\sf{cg-res}\),
            \(\sf{cg-dwr}\),
            \(\sf{dpg}\),
          },
        legend style={at={(0.5,1.2)}, anchor=north, font=\small},
        table/x=ndofs,
        xmax=14000,
        width=12cm,
        height=6cm,
    ]
      \addplot+[table/y expr={\thisrow{hausdorff}/\thisrow{l2_eta}}] table {nine_single/cg_adapt_data.csv};
      \addplot+[table/y expr={\thisrow{hausdorff}/\thisrow{l2_eta}}, restrict x to domain=0:18000] table {nine_single/dwr_adapt_data.csv};
      \addplot+[table/y expr={\thisrow{hausdorff}/\thisrow{l2_eta}}, restrict x to domain=0:25000] table {nine_single/dpg_adapt_data.csv};

    \end{loglogaxis}
    \centering
  \end{tikzpicture}
  \caption{
    Efficiency ratio for the single eigenvalue approximation \(\Lambda_\mathsf{sg} = \{\lambda_9\}\),
    with different discretizations and error estimators
    (\(\bullet \in \{\mathsf{cg-res}, \mathsf{cg-dwr}, \mathsf{dpg}\}\)),
    against the number of \textsc{dofs}.
    (\textsc{dofs} beyond 12,000 have been omitted for clarity.)
  }
  \label{fig:poisson_gww1_single_eff}
\end{figure}
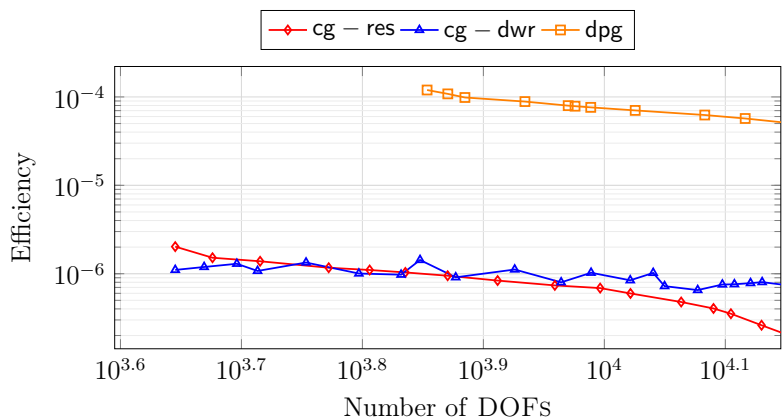

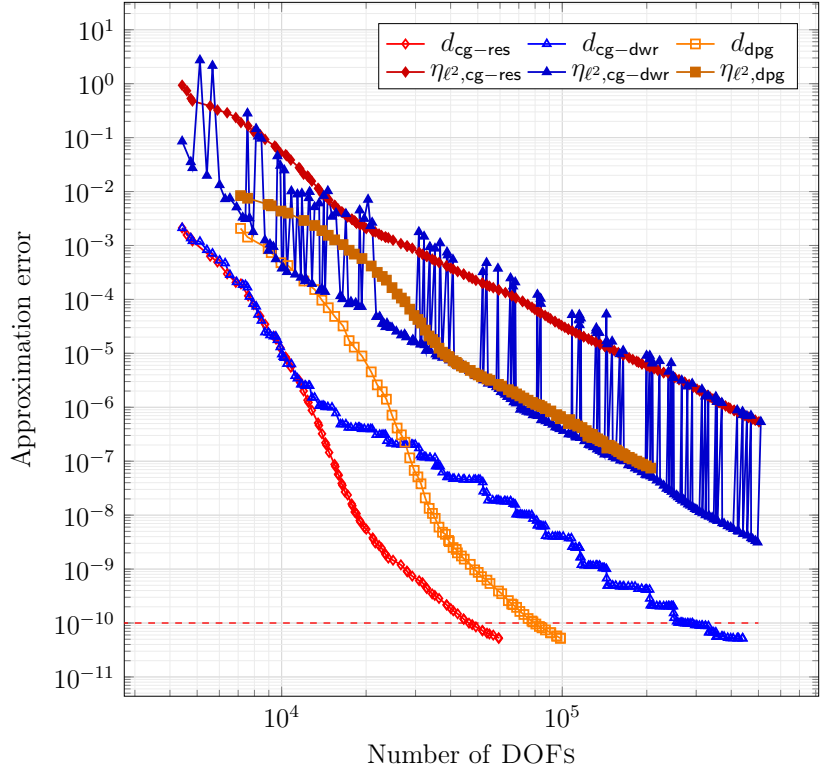
\begin{figure}
  \centering
  \begin{tikzpicture}[scale=0.85]
    \centering
    \begin{loglogaxis}[
        cycle list name={three-by-two},
        xlabel={Number of \textsc{DOFs}},
        ylabel={Approximation error},
        legend columns=3,
        legend entries={%
            \(d_{\sf{cg-res}}\),
            \(d_{\sf{cg-dwr}}\),
            \(d_{\sf{dpg}}\),
            \(\eta_{\ell^2,\sf{cg-res}}\),
            \(\eta_{\ell^2,\sf{cg-dwr}}\),
            \(\eta_{\ell^2,\sf{dpg}}\),
        },
        table/x=ndofs,
    ]
      \addplot+[table/y=hausdorff, y filter/.expression={abs(y) < 0.5e-10 ? NaN : \pgfmathresult }] table {nine_cluster/cg_adapt_data.csv};
      \addplot+[table/y=hausdorff, y filter/.expression={abs(y) < 0.5e-10 ? NaN : \pgfmathresult }] table {nine_cluster/dwr_adapt_data.csv};
      \addplot+[table/y=hausdorff, y filter/.expression={abs(y) < 0.5e-10 ? NaN : \pgfmathresult }] table {nine_cluster/dpg_adapt_data.csv};
      \addplot+[table/y=l2_eta] table    {nine_cluster/cg_adapt_data.csv};
      \addplot+[table/y=l2_eta] table    {nine_cluster/dwr_adapt_data.csv};
      \addplot+[table/y=l2_eta] table    {nine_cluster/dpg_adapt_data.csv};

      \draw[red, dashed, semithick] (axis cs:1e-2,1e-10) -- (axis cs:5e5,1e-10);

    \end{loglogaxis}
    \centering
  \end{tikzpicture}
  \caption{
    Hausdorff distance between the (approximated) cluster of known eigenvalues
    \(\Lambda_\mathsf{cl} = \{\lambda_8, \lambda_9, \lambda_{10}\}\)
    and the computed eigenvalues, \(d_{\bullet}\),
    and \(\ell^2\)-norm of the error estimator
    \(\eta_{\ell^2,\bullet}\), for the different discretizations and error
    estimators
    (\(\bullet \in \{\mathsf{cg-res}, \mathsf{cg-dwr}, \mathsf{dpg}\}\)),
    against the number of \textsc{dofs}.
    The dashed red line indicates the level of accuracy of the reference eigenvalues.
  }
  \label{fig:poisson_gww1_cluster}
\end{figure}

\begin{figure}
  \centering
  \begin{tikzpicture}[scale=0.85]
    \centering
    \begin{loglogaxis}[
        cycle list name={three-by-two},
        xlabel={Number of \textsc{DOFs}},
        ylabel={Efficiency},
        legend columns=3,
        legend entries={%
            \(\sf{cg-res}\),
            \(\sf{cg-dwr}\),
            \(\sf{dpg}\),
        },
        table/x=ndofs,
        xmax=40000,
        legend style={at={(0.5,1.2)}, anchor=north, font=\small},
        width=12cm,
        height=6cm,
    ]
      \addplot+[table/y expr={\thisrow{hausdorff}/\thisrow{l2_eta}}] table {nine_cluster/cg_adapt_data.csv};
      \addplot+[table/y expr={\thisrow{hausdorff}/\thisrow{l2_eta}}, restrict x to domain=0:18000] table {nine_cluster/dwr_adapt_data.csv};
      \addplot+[table/y expr={\thisrow{hausdorff}/\thisrow{l2_eta}}, restrict x to domain=0:25000] table {nine_cluster/dpg_adapt_data.csv};

    \end{loglogaxis}
    \centering
  \end{tikzpicture}
  \caption{
    Efficiency ratio for the cluster eigenvalue approximation \(\Lambda_\mathsf{cl} = \{\lambda_8, \lambda_9, \lambda_{10}\}\),
    with different discretizations and error estimators
    (\(\bullet \in \{\mathsf{cg-res}, \mathsf{cg-dwr}, \mathsf{dpg}\}\)),
    against the number of \textsc{dofs}.
    (\textsc{dofs} beyond 40,000 have been omitted for clarity.)
  }
  \label{fig:poisson_gww1_cluster_eff}
\end{figure}
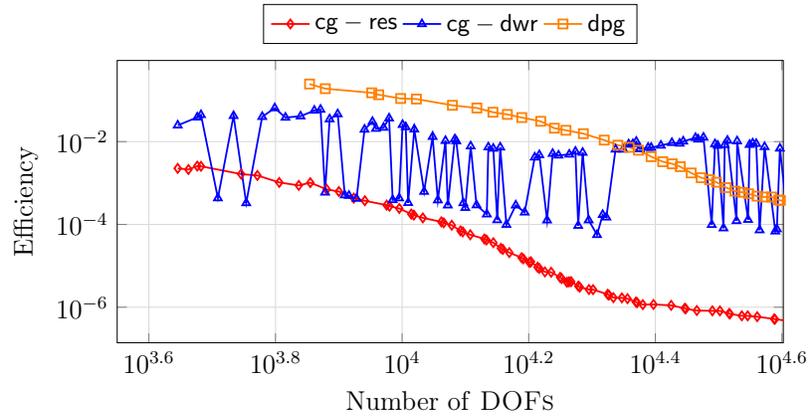

In Figure~\ref{fig:poisson_gww1_cluster}, we show the convergence of the
Hausdorff distance between the (approximated) cluster of known
eigenvalues
\(\Lambda_\mathsf{cl} = \{\lambda_8, \lambda_9, \lambda_{10}\}\) and
their computed approximations \(\Lambda_{h,\mathsf{cl}}\).
Similarly, in Figure~\ref{fig:poisson_gww1_single}, we show the convergence of
the Hausdorff distance between the (approximated) known eigenvalue
\(\Lambda_\mathsf{sg} = \{\lambda_9\}\) and its computed
approximation \(\Lambda_{h,\mathsf{sg}}\).
Recall that the Hausdorff distance between two sets
\(\Upsilon_1, \Upsilon_2 \subset \mathbb{C}\) is defined by
$ d(\Upsilon_1, \Upsilon_2) = \max \{ \sup_{\mu_1 \in \Upsilon_1}
\inf_{\mu_2 \in \Upsilon_2} |\mu_1 - \mu_2|,$
$ \sup_{\mu_2 \in \Upsilon_2} \inf_{\mu_1 \in \Upsilon_1} |\mu_1 -
\mu_2| \}.  $ This distance is indicated by $d$ in
Figures~\ref{fig:poisson_gww1_single}, and~\ref{fig:poisson_gww1_cluster} where the 
legends are subscripted by the discretization method, 
e.g., ``$d_\mathsf{dpg}$'' denotes the method
of Subsection~\ref{ssec:DPG}, ``$d_{\mathsf{cg-dwr}}$'' denotes the
above-mentioned DWR method with the standard Lagrange space
discretization, and
``$d_{\mathsf{cg-res}}$'' denotes the Laplacian-specific explicit residual 
defined by Larson in~\cite{Lar2000-1}.
The plots in the figure show how the  error in eigenvalue cluster approximation
and the 
\(\ell^2\)-norm of the error estimator decays as a function of the
number of degrees of freedom (\textsc{DOFs}) as the mesh gets
adaptively refined.
The eigenvalue cluster error is measured by the Hausdorff distance
between the set of approximate eigenvalues and the previously
mentioned reference eigenvalues from~\cite{Dri1997-1}.  Since the
reference values are expected to be accurate to ten digits after the
decimal point, a level of accuracy indicated in the figures by a
dashed red line, eigenvalue errors lying below this line are not
shown.

The explicit residual error estimator
from~\cite{Lar2000-1} requires less degrees of freedom due to the nature of the
underlying conforming discretization, while the DPG method uses a larger
number of degrees of freedom due to the mixed formulation employed in our
simulations.
This leads to error curves
(in Figures~\ref{fig:poisson_gww1_single} and~\ref{fig:poisson_gww1_cluster})
that appear shifted to the right for the DPG method.
Nonetheless, the DPG error estimator is tighter than the explicit residual
error estimator and the (first-order) DWR error estimator, for both the cluster
and the single eigenvalue cases, as can be observed
in Figures~\ref{fig:poisson_gww1_single_eff} and~\ref{fig:poisson_gww1_cluster_eff}. There,
we plot an easily computable analogue of efficiency
(that uses the eigenvalue cluster errors, which are more easily computed than
the gap between eigenspaces),
denoted by \(\eta_{\ell^2, \bullet}\), defined 
as the ratio between the Hausdorff distance \(d_{\bullet}\) and the
\(\ell^2\)-norm of the error estimator \(\eta_{\ell^2, \bullet}\),
i.e.,
\begin{equation*}
  \text{Eff}_{\ell^2, \bullet}
  = \frac{d_{\bullet}}{\eta_{\ell^2, \bullet}},
\end{equation*}
for \(\bullet \in \{\mathsf{cg-res}, \mathsf{cg-dwr}, \mathsf{dpg}\}\).
Since the estimators $\eta_{\ell^2, \bullet}$ track the eigenspace error,
we have little reason to expect $\text{Eff}_{\ell^2, \bullet}$
to be close to the perfect value of one, yet it is interesting to see
how its values from our experiments.
Note from
Figures~\ref{fig:poisson_gww1_single_eff} and~\ref{fig:poisson_gww1_cluster_eff}
that the DPG efficiency tends to bound the remaining two efficiency
curves from above, indicating a tighter estimator, in the case of the
single eigenvalue (Figure~\ref{fig:poisson_gww1_single_eff}).
We have truncated the plot to a
maximum of 12,000 \textsc{DOFs} to avoid the region where the methods
may achieve greater accuracy than the reference eigenvalue.
In the case of the cluster, the DPG case initially provides a tighter estimator,
but as the number of \textsc{DOFs} increases, the DWR estimator becomes
slightly tighter (Figure~\ref{fig:poisson_gww1_cluster_eff}), yet it oscillates
more than the DPG estimator.
This phenomenon has been observed in our previous
work~\cite{GopGroPinVan2025-1} for a similar (Maxwell) DWR estimator for
a cluster of eigenvalues. 
Similar simulations for the Laplace eigenproblem on the GWW domain
were performed considering the DWR error estimator with second-order derivative
weights (not reported here), yielding efficiency curves
with similar oscillatory behavior.
A more cluster-robust DWR estimator might need to be developed
to overcome this issue (instead of our simple generalization of singleton 
DWR estimator).
In these plots, as in the single eigenvalue case, we have truncated the plot to a
maximum of 40,000 \textsc{DOFs} to avoid the region where the methods may 
achieve greater accuracy than the reference eigenvalues.

\subsection{Helmholtz eigenproblem on a Bragg fiber}

For our second numerical example, we consider the Helmholtz
eigenproblem for a leaky mode. Leaky modes are slightly lossy, yet are
very important practically to guide light through modern
microstructured fibers.  (For some examples of such fibers, see~\cite{GopGroPinVan2025-1}.)  These modes satisfy an outgoing
boundary condition at infinity permitting the leakage of energy which
makes the corresponding eigenproblem nonselfadjoint. Here we consider
a simple model fiber, the Bragg fiber, which is often used to get
intuition for more complex fibers. We handle the outgoing condition by
inserting a perfectly matched layer (PML)~\cite{Ber1994-1,
  ColMon1998-1, BraPas2007-1} which modifies the solution within that layer to a
decaying one (while leaving the solution unchanged in the remainder) allowing us to truncate the infinite-domain problem to a bounded
computational domain~$\Omega$. The purpose of this subsection is to present a
computational study of the performance of the proposed DPG error
estimator for this nonselfadjoint eigenproblem.  (A full verification
of Assumption~\ref{asm:resolvent-approx} for this case requires
further advances in DPG theory which will take us too far away from
the present focus.)

The equation for the Helmholtz eigenproblem on the transverse fiber
cross section, after subtracting off the homogenous refractive index
$n_0$ of the infinite air surround, and after a nondimensionalization
by a characteristic length scale~$L$, becomes (see
e.g.,~\cite[Eq.~(21a--b)]{GopParVan2022-1})
\begin{equation}
  \label{eq:before-pml}
\begin{aligned}
  - \Delta u + V u & = Z^2 u, \qquad \text{ in } \mathbb{R}^2, \\
  u \text{ is outgoing } & \text{ as } r = \sqrt{x_0^2 + x_1^2} \to \infty,
\end{aligned}  
\end{equation}
where \(V\) is the \emph{index well} defined by
\(V(x) = L^2 k^2 (n_0^2 - n^2(x))\), \(k\in \mathbb{R}\) is the
operational wavenumber, \(n = n(x)\) is the refractive index profile
of the
fiber, 
and \(Z^2\) is the nondimensionalized eigenvalue to be found.  In
terms of the physical ``propagation constant'' $\beta$, 
the original Helmholtz eigenvalue is  $\beta^2$, and  the nondimensional eigenvalue
is $Z^2 = L^2 (k^2 n_0^2 - \beta^2)$.

To truncate the domain using a PML, we perform a complex change of
variables \(\tilde{x} = \Phi(x)\) where
\(\Phi(x) := \frac{\eta(r)}{r} x\) with complex-valued function
\(\eta: \mathbb{R}_{\geq 0} \to \mathbb{C}\), defined shortly, of the
radial coordinate $r$.  The map $\Phi$ is the identity on the disk
$r < r_0$, while for $r>r_0$ it is designed to transform the solution
to an exponentially decaying function of~$r$.
It is standard (see e.g.,~\cite{ColMon1998-1}) to transform the
problem~\eqref{eq:before-pml} to \(\mathcal{A} u = Z^2 u\), where the
operator $\mathcal{A}$ now depends on $\Phi$, and then truncate the domain to
a finite domain $\om$ of some radius sufficiently greater than $r_0$
where  the transformed mode~$u$ has decayed enough to be indistinguishable from zero.
Using this
$\mathcal A$, the following modified DPG discretization of the
resolvent problem \((Z^2 - \mathcal{A}) u = f\) on $\Omega$ is obtained:
find \(u_h \in V_h\), \(\hat{q}_n \in \hat{X}_h\) such that
\begin{align}
  (\varepsilon_h, \delta_h)_{H^1(\Omega_h)} +
  a_{Z^2,\mathsf{PML}}((u_h, \hat{q}_n), \delta_h) &
  = (\det(J) f, \delta_h)_{L_2(\Omega)}, \\
  \overline{a_{Z^2,\mathsf{PML}}((w_h, \hat{r}_n), u_h)} &
  = 0,
\end{align}
for all \(\delta_h \in Y_h\) and \((w_h, \hat{r}_n) \in V_h \times \hat{X}_h\),
where
\begin{equation}\label{eq:a-z2-pml}
  \begin{aligned}
    a_{Z^2,\mathsf{PML}}((u, \hat{q}_n), v) 
    =
    & Z^2(\det(J) u, v)_{L_2}
      - (\gamma \nabla u, \nabla v)_{L_2}
    \\
    & - (\det(J) V u, v)_{L_2}
      + \langle \det(J) J^{-\top} \hat{q}_n, v \rangle_h    
  \end{aligned}
\end{equation}
for all \(v \in Y_h\),
where \(J\) denotes the Jacobian matrix of the transformation \(\Phi\), and
\(\gamma = J^{-1} J^{-\top} \det(J)\).

There are multiple ways to choose the mapped radius $\eta(r)$, as previously
observed in the literature~\cite{ColMon1998-1, GopParVan2022-1, KimPas2009-1, NanWes2018-1}.
We use a two-dimensional analogue of an expression in~\cite{KimPas2009-1}
(see also~\cite{GopGroPinVan2025-1}).
For simplicity, fix $\Omega$ to be a disk of radius $r_1$.
Furthermore, assume that support of the index well
$V$ is contained in a disk of radius
$r_0 < r_1$, and that a cylindrical PML is set in the annular region
\(r_0 < r < r_1\). 
Let \(0 < \alpha\) be the PML strength parameter. 
We set 
\begin{subequations}\label{eq:eta-defn}
\begin{equation}
  \eta(r)  = r(1 + \ii \phi(r)),
\end{equation}
where 
\begin{equation}
  \phi(r) =
  \begin{cases}
    0, & \text{ if } r < r_0, \\
    \alpha\, \displaystyle{
      \frac{\displaystyle{\int_{r_0}^r (s - r_0)^{2} (s - r_1)^{2} \,
      \mathrm{d}s}}{
        \displaystyle{
          \int_{r_0}^{r_1} (s - r_0)^{2} (s - r_1)^{2} \, \mathrm{d}s}
        },
      }
    & \text{ if } r_0 < r < r_1.
  \end{cases}
\end{equation}
\end{subequations}
This together with  $\zeta(r) = \eta(r) / r$ defines all quantities required
to compute the Jacobian \(J\) of the transformation \(\Phi\).

For the purpose of this example, we consider a simple Bragg fiber
(see~\cite[Section 4]{GopGroPinVan2025-1}, and references therein),
with coefficients defined by Table~\ref{tab:bragg-fiber-params}.  The
geometry of the fiber is illustrated in Figure~\ref{fig:geo_bragg}.
The initial mesh in Figure~\ref{fig:mesh_bragg_init} is a
quasi-uniform mesh with approximate element size \(h \approx 1.0\) for
the PML region, and \(h \approx 0.1\) for the inner fiber region.  We
approximate an exact eigenvalue of multiplicity two, i.e.,  
the exact cluster $\varLambda$ is a singleton containing
\(Z^2 = 2.0302671 + \ii\, 0.0024486\), a reference
value obtained by semianalytical computations.  (The corresponding two
eigenfunctions have patterns similar to the standard two-leaved
linearly polarized LP\textsubscript{11}-like guided modes within the core,
but they are leaky modes.)  The associated discrete cluster \(\varLambda_h\)
is computed using the FEAST method, with a circular contour \(\Gamma\)
centered at the exact eigenvalue and radius \(r_\Gamma = 0.1\), and usually has two elements, both close to the exact eigenvalue.

Figure~\ref{fig:mesh_bragg_final} displays the final mesh after
running the adaptive algorithm using the DPG error estimator for
approximating the second fundamental eigenvalue cluster.
Observe that the refinement is concentrated in the glass cladding of the
fiber.
This is in concordance with the oscillatory fine structures observed
in~\cite{VanGopGro2023-1}, and depicted in Figures~\ref{fig:mode1_rescaled_bragg}
and~\ref{fig:mode2_rescaled_bragg}, where the scalar components of the
eigenmodes are rescaled to better visualize the behavior in the glass ring.
Figures~\ref{fig:mode1_bragg} and~\ref{fig:mode2_bragg} display the same
eigenmodes without
rescaling, showing how the modes are mostly confined to the
air core of the fiber.
The performance of the DPG error estimator is illustrated in
Figure~\ref{fig:bragg_dpg_convergence}, where we plot the convergence of the
Hausdorff distance between the exact eigenvalue cluster \(\varLambda\) and
its computed approximation \(\varLambda_h\), as well as the
\(\ell^2\)-norm of the DPG error estimator, against the number of degrees of
freedom.

\begin{table}
  \centering
  \begin{tabular}{r r r}
    \toprule
    Name & Symbol & Value \\
    \midrule
    Nondimensionalization length scale & \(L\) & \(1.5 \times 10^{-5}\)~m \\
    Operating wavenumber & \(k\) & \(\frac{2 \pi}{1.7} \times 10^{6}~\text{m}^{-1}\) \\
    Refractive index in air regions & \(n_{\rm air}\) & \(1.00027717\) \\
    Refractive index in glass ring & \(n_{\rm glass}\) & \(1.43881648\) \\
    Air core radius & \(r_{\rm core}\) & \(2.7183\) \\
    Outer radius of glass ring & \(r_{\rm outer}\) & \(3.385\) \\
    Thickness of glass ring & \(t_{\rm ring}\) & \(0.66666667\) \\
    Outer radius of outer air region & \(r_{0}\) & \(4.385\) \\
    Thickness of outer air region & \(t_{\rm air}\) & \(1.0\) \\
    Outer radius of PML & \(r_{1}\) & \(8.05166666\) \\
    Thickness of PML & \(t_{\rm PML}\) & \(3.66666667\) \\
    PML strength & \(\alpha\) & \(10.0\) \\
    \bottomrule
  \end{tabular}
  \caption{Parameters defining the Bragg fiber used in the numerical
  experiments.}
  \label{tab:bragg-fiber-params}
\end{table}
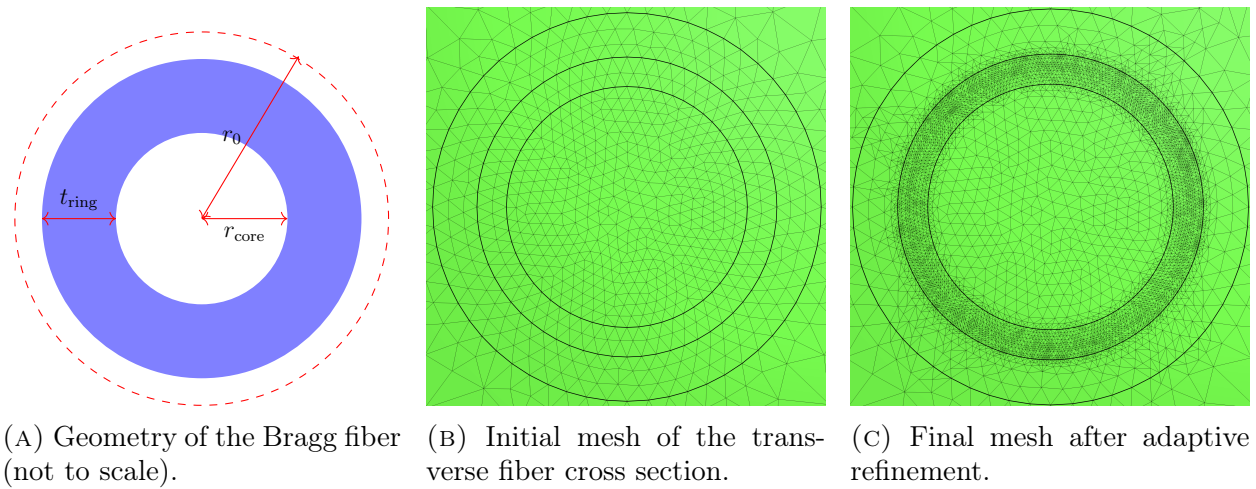
\begin{figure}
  \centering
  \begin{subfigure}[t]{0.32\textwidth}
    \centering
    \begin{tikzpicture}[scale=0.03]
      \fill [blue, semitransparent, even odd rule] (0,0)
      circle[radius=70.8]
      circle[radius=38.0];
      \draw[red, <->] (0,0) -- (38.0,0)
      node [scale=0.75,black, midway, below]
        {$r_{\text{core}}$} ;
      \draw[red, <->] (-70.8,0) -- (-38.0,0)
      node [scale=0.75,black, midway, above]
      {$t_{\text{ring}}$} ;
      \draw[red, dashed] (0,0) circle [radius=82.8];
      \draw[red, <->] (0,0) -- (43.0,72.0)
      node [scale=0.75,black, midway, left]
      {$r_0$} ;
    \end{tikzpicture}
    \caption{Geometry of the Bragg fiber (not to scale).}
    \label{fig:geo_bragg}
  \end{subfigure}%
  \hfill
  \begin{subfigure}[t]{0.32\textwidth}
      \centering
      \includegraphics[width=\textwidth]{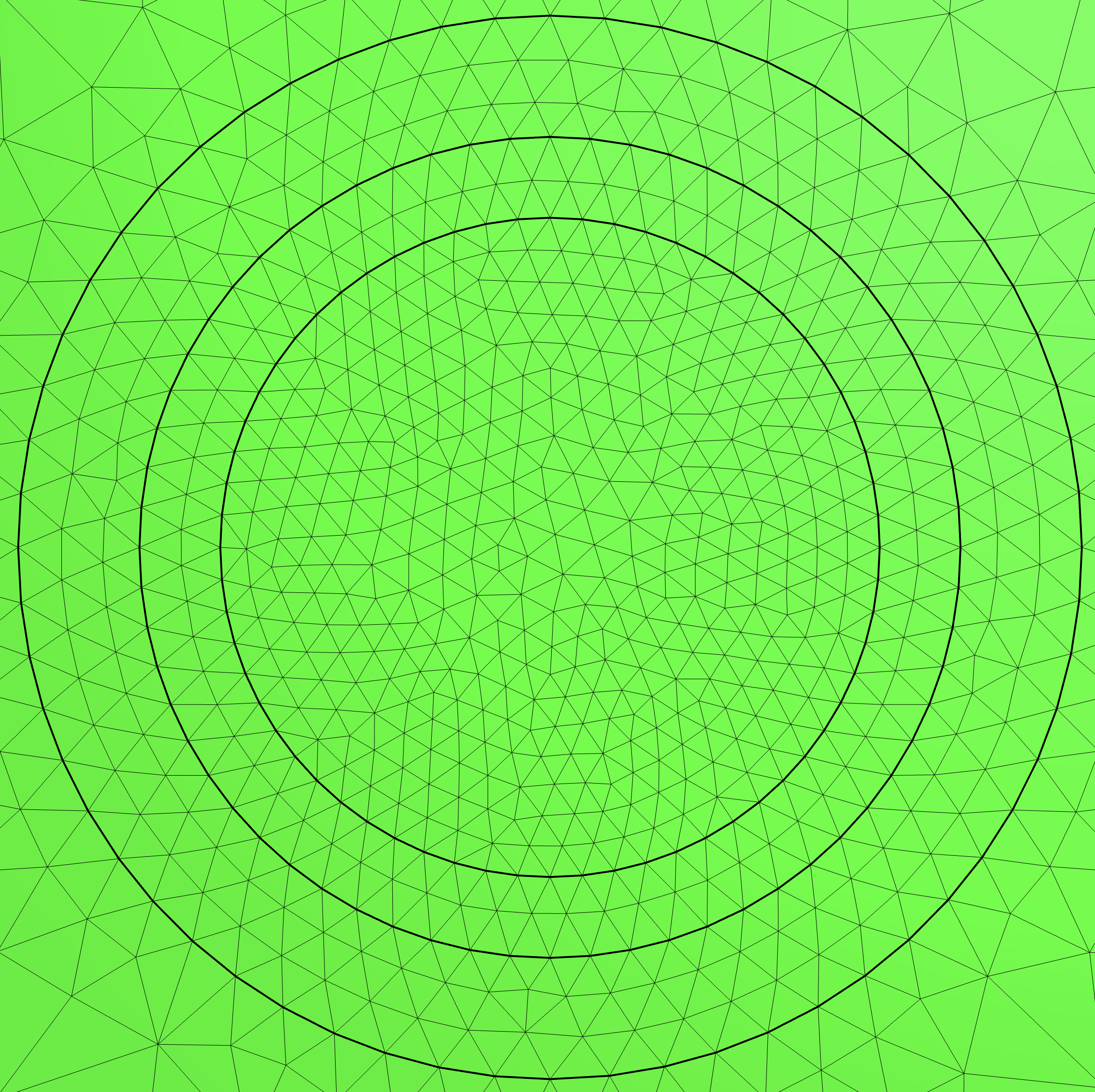}
      \caption{Initial mesh of the transverse fiber cross section.}
      \label{fig:mesh_bragg_init}
  \end{subfigure}%
  \hfill
  \begin{subfigure}[t]{0.32\textwidth}
      \centering
      \includegraphics[width=\textwidth]{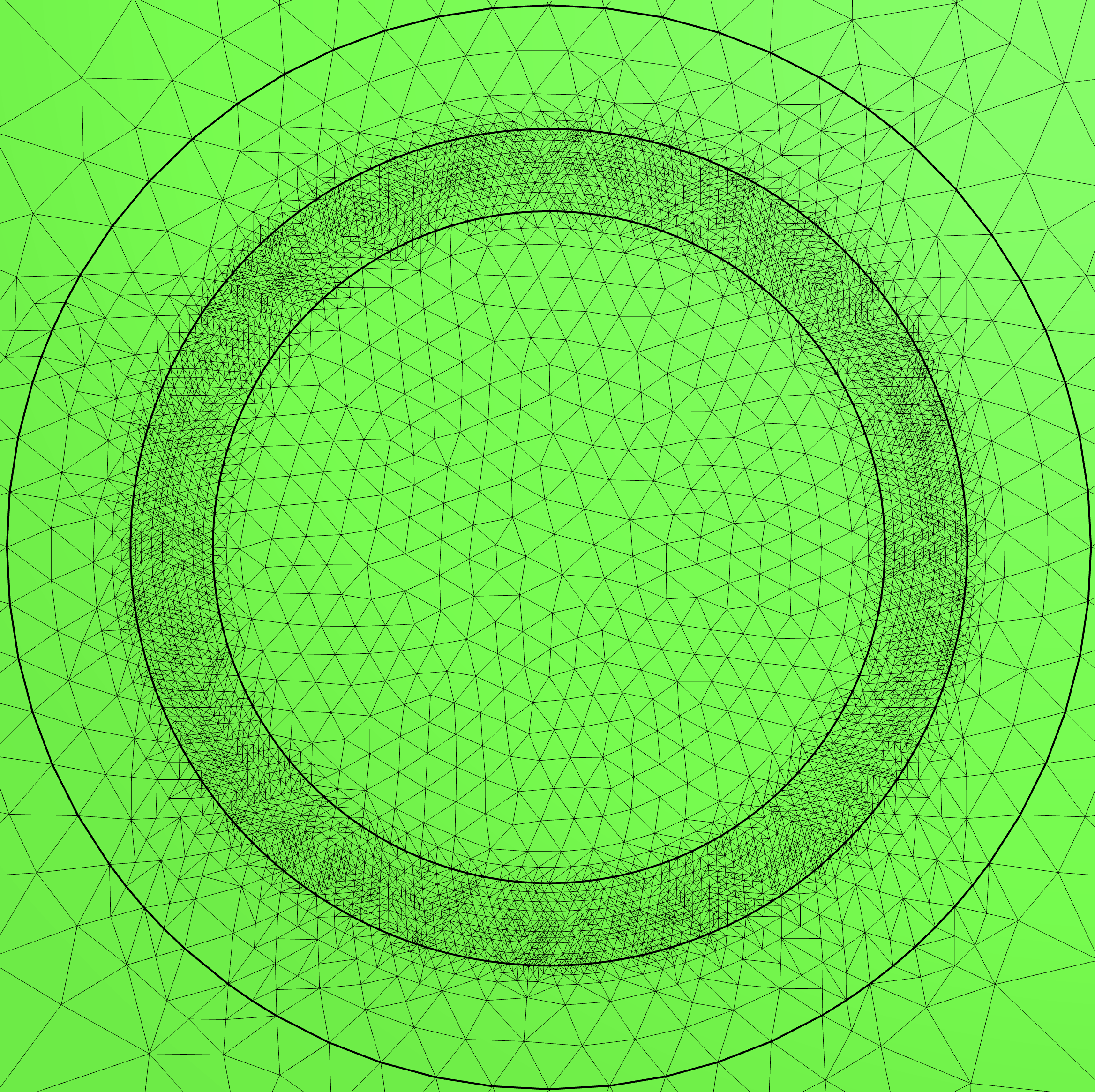}
      \caption{Final mesh after adaptive refinement.}
      \label{fig:mesh_bragg_final}
  \end{subfigure}%
    \caption{
      Geometry of the Bragg fiber (left),
      initial mesh (center), and final mesh after adaptive refinement (right)
      using the DPG error estimator for approximating the second fundamental
      eigenvalue cluster.
    }
\end{figure}

\begin{figure}
    \begin{subfigure}[t]{0.48\textwidth}
        \centering
        \includegraphics[width=\textwidth]{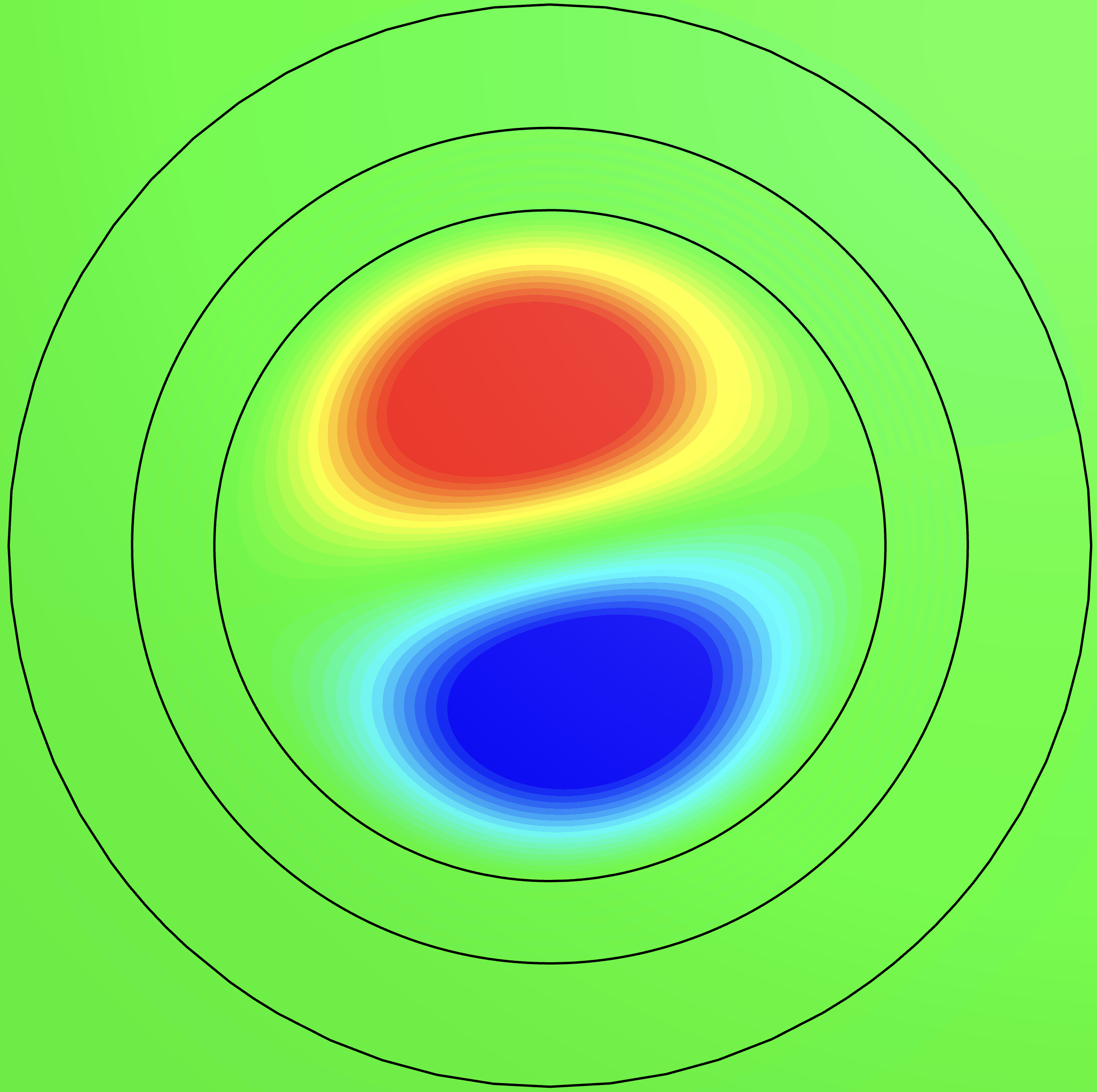}
        \caption{Scalar component of the first mode in the cluster.}
        \label{fig:mode1_bragg}
    \end{subfigure}%
    \hfill
    \begin{subfigure}[t]{0.48\textwidth}
        \centering
        \includegraphics[width=\textwidth]{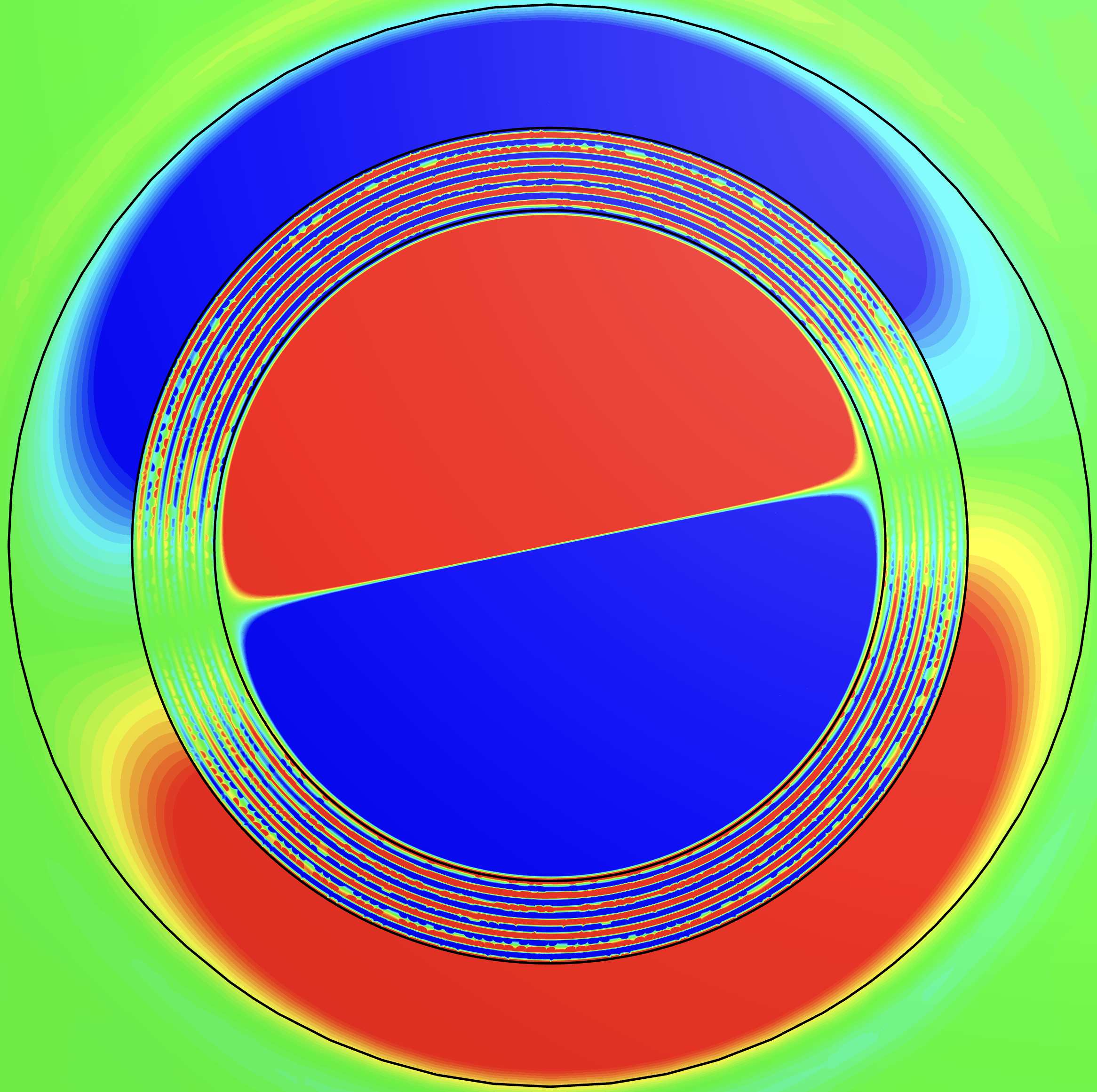}
        \caption{Rescaled scalar component of the first mode in the cluster.}
        \label{fig:mode1_rescaled_bragg}
    \end{subfigure}%
    \hfill
    \begin{subfigure}[t]{0.48\textwidth}
        \centering
        \includegraphics[width=\textwidth]{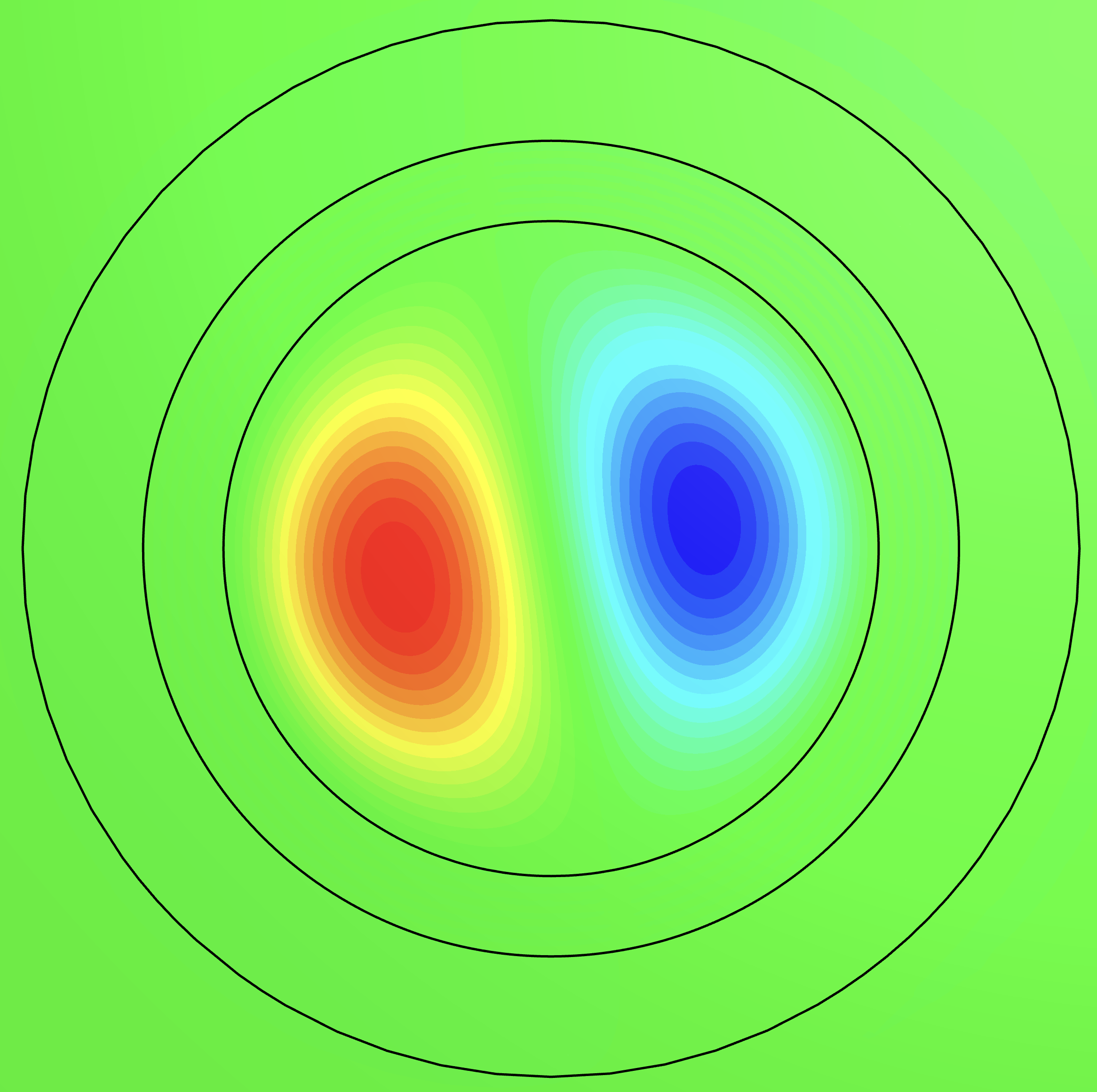}
        \caption{Scalar component of the second mode in the cluster.}
        \label{fig:mode2_bragg}
    \end{subfigure}%
    \hfill
    \begin{subfigure}[t]{0.48\textwidth}
        \centering
        \includegraphics[width=\textwidth]{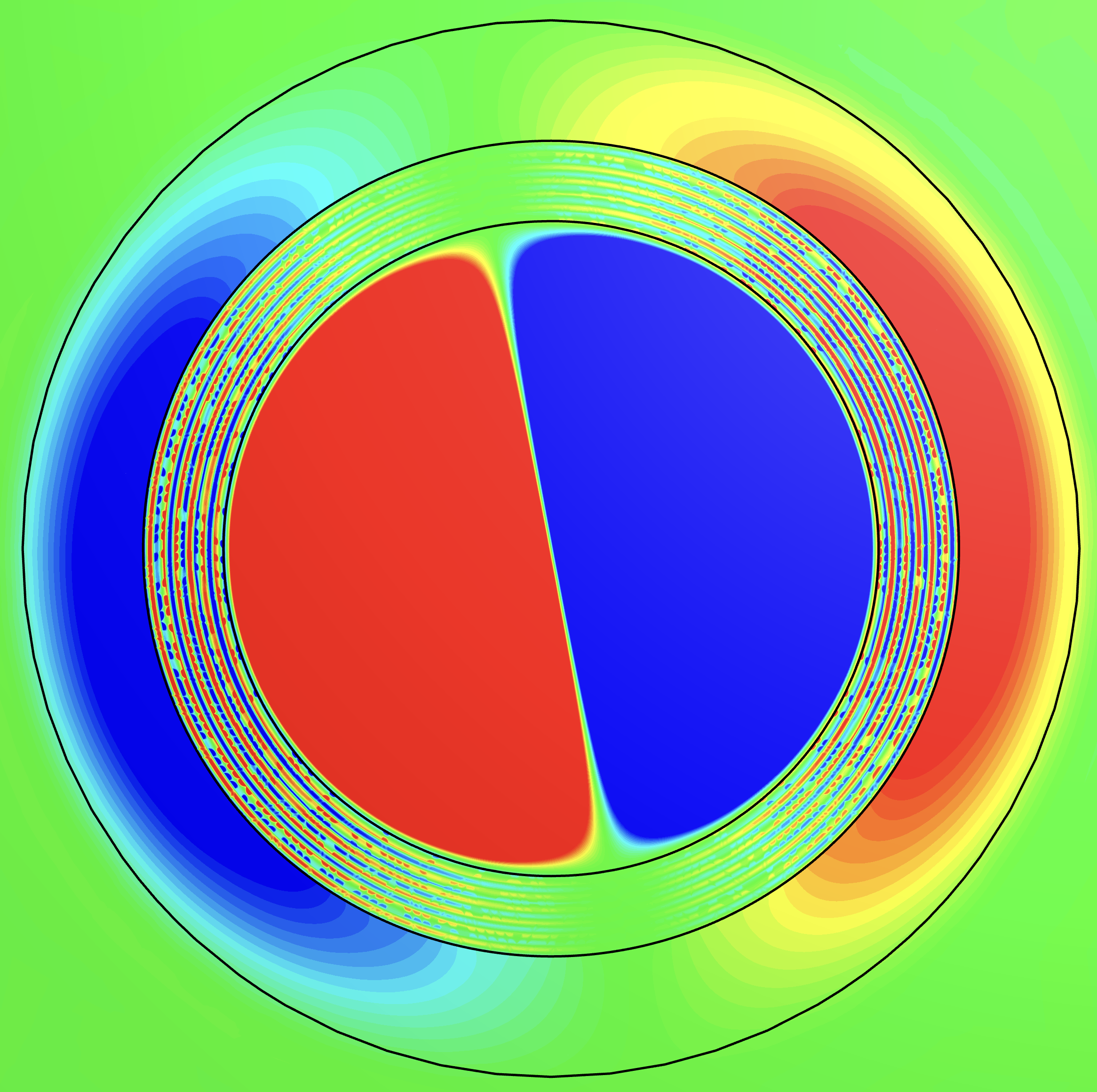}
        \caption{Rescaled scalar component of the second mode in the cluster.}
        \label{fig:mode2_rescaled_bragg}
    \end{subfigure}%
    \caption{
      Two Bragg eigenmodes of the computed cluster are shown in
      different color scales in the left column and right column. The color scale of right column is selected for better visualization of the 
      oscillatory features in the cladding region.
    }
    \label{fig:modes_bragg}
\end{figure}

\begin{figure}
  \centering
  \begin{tikzpicture}[scale=0.85]
    \centering
    \begin{loglogaxis}[
        cycle list name={one-by-two},
        xlabel={Number of \textsc{DOFs}},
        ylabel={Approximation error},
        legend columns=3,
        legend entries={%
            \(d_{\sf{dpg}}\),
            \(\eta_{\ell^2,\sf{dpg}}\),
          },
        height=10cm,                  
        table/x=ndofs,
    ]
      \addplot+[table/y=hausdorff_dist] table {bragg_dpg/lkm_dpg_adapt_results.csv};
      \addplot+[table/y=l2_eta] table {bragg_dpg/lkm_dpg_adapt_results.csv};


    \end{loglogaxis}
    \centering
  \end{tikzpicture}
  \caption{
    Hausdorff distance between the (approximated) eigenvalue
    cluster \(\Lambda = \{Z^2\}\) and its computed
    approximation, \(d_{\mathsf{dpg}}\),
    and \(\ell^2\)-norm of the DPG error estimator
    \(\eta_{\ell^2,\mathsf{dpg}}\), against the number of \textsc{dofs}.
  }
  \label{fig:bragg_dpg_convergence}
\end{figure}
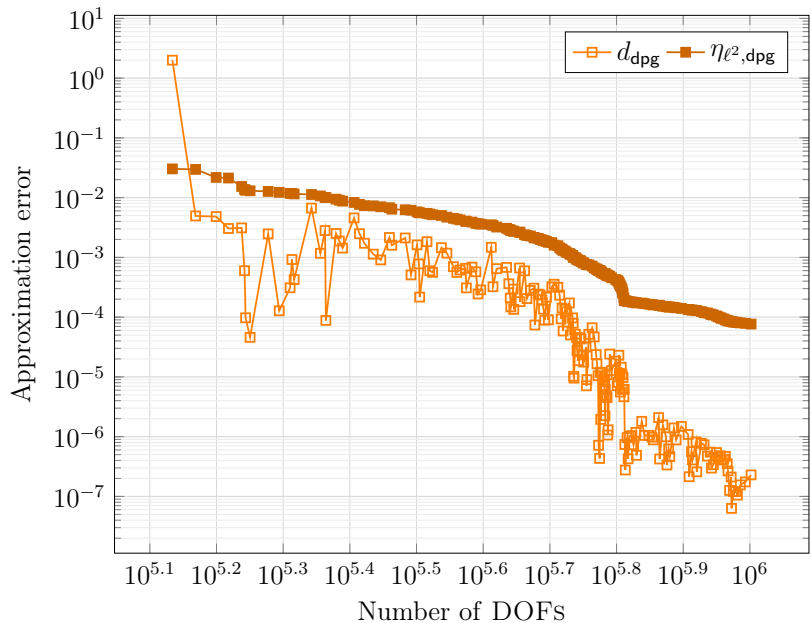



\subsection*{Acknowledgments}

This work was supported in part by the NSF grant DMS-2245077. 
It also benefited from activities organized under the auspices of NSF
RTG grant DMS-2136228.

\newpage 

\bibliographystyle{siam}
\bibliography{refs.bib}

\end{document}